\documentclass{amsart}
\usepackage{amsfonts,amsmath,amsthm,amssymb,tikz,xcolor,mathtools,extarrows,graphicx,stackrel}
\usepackage[normalem]{ulem}
\useunder{\uline}{\ul}{}
\usepackage[margin=1in]{geometry}
\usepackage[colorlinks=true,linkcolor=red!70!black,citecolor=green!70!black,urlcolor=magenta!70!black,backref]{hyperref}
\usepackage{colortbl}
\usepackage{hhline}
\usepackage{xcolor}
\usepackage{dsfont}
\usepackage{standalone}
\usepackage{comment}
\usepackage{wasysym} 
\usepackage{bbm} 
\usepackage{stmaryrd} 
\usepackage{enumitem}
\usepackage{xargs}
\usepackage{color}
\usepackage[all, color, cmtip]{xy}

\colorlet{darkblue}{blue!70!black}
\colorlet{darkred}{red!70!black}
\colorlet{darkgreen}{green!70!black}
\colorlet{darkwhite}{white!65!black}
\colorlet{darkorange}{orange!70!black}

\colorlet{darkmagenta}{magenta!70!black}
\colorlet{pink}{green!20!magenta}

\colorlet{lightcyan}{cyan!15!white}
\colorlet{lightyellow}{yellow!30!white}
\colorlet{lightcyanb}{lightcyan!50!black}

\definecolor{purple}{rgb}{0.63, 0.36, 0.94} 

\usetikzlibrary{decorations.markings,decorations.pathreplacing, decorations.pathmorphing, calligraphy,positioning,calc,arrows.meta,cd}

\numberwithin{equation}{section}

\newcommand{\xsum}[2]{
  \xy
  (0,.4)*{\displaystyle\sum};
  (0,-6)*{\scs #2};
  (0,-3.7)*{\scs #1};
  \endxy
}

\newtheorem{thm}{Theorem}[section]
\newtheorem{lemma}[thm]{Lemma}
\newtheorem{proposition}[thm]{Proposition}

\theoremstyle{definition}
\newtheorem{definition}[thm]{Definition}

\newtheorem{remark}[thm]{Remark}

\usepackage{bbm}

\def\Z{{\mathbbm Z}}

\def\mf{\mathfrak}

 \def\1{\mathbbm{1}}%
 
\def\Id{\mathrm{Id}}
\def\Q{{\mathbbm Q}}

\newcommand{\scs}{\scriptstyle}

\def\cal{\mathcal}
\def\la{\langle}
\def\ra{\rangle}
\newcommand{\sE}{\mathcal{E}}
\newcommand{\sF}{\mathcal{F}}

\newcommand{\onel}{\1_{\lambda}}

\newcommand{\Ucat}{\mathcal{U}_Q}

\newcommand{\Hom}{{\rm Hom}}

\renewcommand{\to}{\rightarrow}
\newcommand{\maps}{\colon}
\newcommand{\END}{{\rm END}}
\newcommand{\elem}{e}
\newcommand{\Sym}{\operatorname{Sym}}
\newcommand{\witt}{{\mathfrak{W}}}
\newcommand{\LLn}[1][n]{{\mathsf{L}_{#1}}}
\newcommand{\ourwitt}{{\witt_{-1}^{\infty}}}
\renewcommand{\l}{\lambda}

\newif\ifshownode
\def\checkarg#1{
  \def\tempa{#1}\def\tempb{0}%
  \ifx\tempa\tempb%
    \shownodefalse%
  \else%
    \shownodetrue%
  \fi%
}

\newcommand{\hackcenter}[1]{
 \xy (0,0)*{#1}; \endxy}
\newcommandx{\CIRCLECUPweb}[7][1 = 0.5, 2 = <, 3 = 1, 4 = $\lambda$,5 = 3, 6 = 1, 7 = <]{\quad
\xy 
(0,0)*{
  \begin{tikzpicture}[, scale=#1]
\checkarg{#5}\ifshownode
    \draw[postaction={decoration={markings,
                        mark=at position 0.6 with {\arrow{#2}};    }, decorate}, color=green, ultra thick ] (0.5,1.25)arc (0:360:.75);
                        \node at  (-0.25, 2.25)[top, color = green]{\scriptsize $#5$};
    \else
    \draw[postaction={decoration={markings,
                        mark=at position 0.6 with {\arrow{#2}};    }, decorate}, thick ] (0.5,1.25)arc (0:360:.75);
    \fi
    \draw[postaction = { decoration={markings,mark=at position 0.15 with {\arrow{#7}},  mark=at position 0.85 with {\arrow{#7}}}, decorate},domain=-2:2, thick]  plot (\x, 0.5*\x*\x);
    \node at (0.5, 1.25) {\(\star\)};
    \node at  (0.5, 1.25)[right]{\scriptsize $#3$};
    \node at (2,0.75) {\scriptsize $#4$};
     \checkarg{#6}\ifshownode
    \node at (0, 0) {$\bullet$};
    \node at (0,0)[below]{\scriptsize $#6$};
    \fi
  \end{tikzpicture}}
\endxy
}

\newcommandx{\Xweb}{%
\xy 0;/r.30pc/:
  (0,0)*{\xybox{%
    (-4,-4)*{};(4,4)*{} **\crv{(-4,-1)&(4,1)}?(1)*\dir{>};
    (4,-4)*{};(-4,4)*{} **\crv{(4,-1)&(-4,1)}?(1)*\dir{>};
    (9,1)*{\scs \lambda};
    (-10,0)*{};(10,0)*{};
  }};
\endxy
}

\newcommandx{\TWOweb}[0]{\xy
(0,0)*{ \begin{tikzpicture}[scale=0.6]
    \draw[thick, ->]  (-1,-1) to (-1,1);
    \draw[thick, ->]  (0,-1) to (0,1);
    \node at (.8,-.4) {\scriptsize $\lambda$};
\end{tikzpicture}} \endxy}

\newcommandx{\CIRCLEweb}[7][1 = 0.5, 2 = <, 3 = 1, 4 = 0, 5 = black, 6 = 0, 7 = i]{
\xy 
(0,0)*{
  \begin{tikzpicture}[decoration={markings,
                        mark=at position 0.6 with {\arrow{#2}};    }, scale=#1]
    \checkarg{#6}\ifshownode
    \draw[postaction={decorate}, color=#5, ultra thick ] (0,0)arc (0:360:.75);
    \else
    \draw[postaction={decorate}, color=#5, thick ] (0,0)arc (0:360:.75);
    \fi
    \node at (-.75, -.75) {\(\star\)};
    \node at (-0.75,-0.75)[below]{\scriptsize $#3$};
    \checkarg{#4}\ifshownode
            \node at (0,0.75) {\scriptsize $#4$};
    \fi
    \checkarg{#7}\ifshownode
        \node at (-1.5,0.75) {\tiny $#7$};
    \fi
  \end{tikzpicture}}
\endxy
}

\newcommandx{\CUPweb}[7][1 = 1, 2 = <, 3 = 0, 4 = 0, 5 = black, 6 = 0, 7 = $i$]{
\xy 
(0,0)*{
  \begin{tikzpicture}[decoration={markings,
                     mark=at position 0.3 with {\arrow{#2}}, mark=at position 0.95 with {\arrow{#2}};  }, scale=#1]
    \checkarg{#6}\ifshownode
    \draw[postaction={decorate}, color=#5, ultra thick ](.75,2) .. controls ++(0,-.75) and ++(0,-.75) .. (0,2);
    \else
    \draw[postaction={decorate}, color=#5, thick] (.75,2) .. controls ++(0,-.75) and ++(0,-.75) .. (0,2);
    \fi
    \checkarg{#3}\ifshownode
    \node at (.4, 1.45) { $\bullet$};
    \node at (.4, 1.45)[below]{\scriptsize $#3$};
    \fi
    \checkarg{#4}\ifshownode
         \node at (-.2, 1.45) {\scriptsize $#4$};
        \else \node at (-.2, 1.45) {\scriptsize $\lambda$};
    \fi
     \checkarg{#7}\ifshownode
    \node at (.95,1.9) {\tiny $i$};
    \fi
  \end{tikzpicture}}
\endxy
}

\newcommandx{\Xwebwithtext}[8][1 = 0.8, 2 = 0, 3 = 0, 4 = 0, 5 = 0, 6 = $i$, 7 = $i$, 8 = $\lambda$]{
\xy
(0,0)*{
\begin{tikzpicture}[scale=#1]
    \draw[thick, ->] (0,0) .. controls ++(0,.55) and ++(0,-.5) .. (.75,1);
    \draw[thick, ->] (.75,0) .. controls ++(0,.5) and ++(0,-.5) .. (0,1);
    \checkarg{#2}\ifshownode
    \node at (.125, .3) {$\bullet$};
    \node at (.125, .3)[left] {\scriptsize #2};
    \fi
    \checkarg{#3}\ifshownode
    \node at (.625, .3) {$\bullet$};
   \node at (.625, .3)[right] {\scriptsize #3};
    \fi
    \checkarg{#4}\ifshownode
    \node at (.11, .7) {$\bullet$};
    \node at (.11, .7)[left] {\scriptsize #4};
    \fi
\checkarg{#5}\ifshownode
    \node at (.64, .7) {$\bullet$};
    \node at (.64, .7)[right] {\scriptsize #5};
    \fi
    \checkarg{#8}\ifshownode
    \node at (1.1,.5) {\scriptsize #8};
    \fi
    \node at (-.2,.10) {\tiny #6};
    \node at (.95,.10) {\tiny #7};
\end{tikzpicture}}
\endxy}

\newcommandx{\CUPwebwithtext}[7][1 = 1, 2 = <, 3 = 1, 4 = 0, 5 = black, 6 = 0, 7 = 0]{
\xy 
(0,0)*{
  \begin{tikzpicture}[decoration={markings, mark=at position 0.3 with {\arrow{#2}};  }, scale=#1]
    \checkarg{#6}\ifshownode
    \draw[postaction={decorate}, color=#5, ultra thick ] (1,2) .. controls ++(0,-.75) and ++(0,-.75) .. (-0.25,2);
    \else
    \draw[postaction={decorate}, color=#5, thick] (1,2) .. controls ++(0,-.75) and ++(0,-.75) .. (-0.25,2);
    \fi
    \checkarg{#3}\ifshownode
    \node at (.4, 1.45) {$\bullet$};
    \node at (.4, 1.45)[below]{\scriptsize $#3$};
    \fi
    \checkarg{#4}\ifshownode
        \node at (-.2, 1.45) {\scriptsize $#4$};
        \else \node at (-.2, 1.45) {\scriptsize $\lambda$};
    \fi
    \node at (1.2,1.9) {\tiny $i$};
        \checkarg{#7}\ifshownode
        \node at (0.25, 1.9) {\tiny $#7$};
        \else \node at (0.25, 1.9) {\tiny $p_{i,r}(\lambda)$};
    \fi
  \end{tikzpicture}}
\endxy
}

\newcommandx{\CAPweb}[6][1 = 1, 2 = <, 3 = 0, 4 = $\lambda$, 5 = black, 6 = 0]{
\xy 
(0,1)*{
  \begin{tikzpicture}[decoration={markings,
                     mark=at position 0.3 with {\arrow{#2}}, mark=at position 0.95 with {\arrow{#2}};  }, scale=#1]
    \checkarg{#6}\ifshownode
    \draw[postaction={decorate}, color=#5, ultra thick ] (.75,-2) .. controls ++(0,.75) and ++(0,.75) .. (0,-2);
    \else
    \draw[postaction={decorate}, color=#5, thick] (.75,-2) .. controls ++(0,.75) and ++(0,.75) .. (0,-2);
    \fi
    \checkarg{#3}\ifshownode
    \node at (.4,-1.45) {$\bullet$};
    \node at (.4,-1.45)[above]{\scriptsize #3};
    \fi
     \node at (-.2,-1.45) {\scriptsize #4};
    \node at (.95,-1.9) {\tiny $i$};
  \end{tikzpicture}}
\endxy
}

\newcommand{\iccbub}[2]{
\xybox{%
 (-6,0)*{};
  (6,0)*{};
  (-4,0)*{}="t1";
  (4,0)*{}="t2";
  "t2";"t1" **\crv{(4,6) & (-4,6)}; ?(.7)*\dir{}+(-2,0)*{\scs #2}
  ?(.05)*\dir{>} ?(1)*\dir{>};
  "t2";"t1" **\crv{(4,-6) & (-4,-6)};
   ?(.3)*\dir{}+(0,0)*{\star}+(0,-3)*{\scs {#1}};
}}

 \newcommandx{\CIRCLELEFTweb}[1][1 = \lambda-1]{\xy 0;/r.18pc/:
 (0,0)*{\icbub{#1}{}};
 \endxy}

 \newcommandx{\CIRCLERIGHTweb}[1][1 = \lambda-1]{\xy 0;/r.18pc/: 
 (0,0)*{\iccbub{#1}{}};
 \endxy}

\newcommand{\ticbub}[2]{%
  \xybox{%
    (-4,0)*{}="t1";%
    (4,0)*{}="t2";%
    {\xy
      "t2";"t1" **@{->}@/_1.0pc/@*{[green]@{|}} ?(.7)*+<0pt,-4pt>{\scs #2};
    \endxy}%
    {\xy
      "t2";"t1" **@{->}@/^1.0pc/@*{[green]@{|}} ?(.3)*+<0pt,-4pt>{\star}+(0,-3)*+<0pt,-4pt>{\scs {#1}};
    \endxy}%
  }%
}

\newcommand{\ticcbub}[2]{
\xybox{%
 (-6,0)*{};
  (6,0)*{};
  (-4,0)*{}="t1";
  (4,0)*{}="t2";
  "t2";"t1" **\crv{(4,6) & (-4,6)}; ?(.7)*\dir{}+(-2,0)*{\scs #2}
  ?(.05)*\dir{>} ?(1)*\dir{>};
  "t2";"t1" **\crv{(4,-6) & (-4,-6)};
   ?(.3)*\dir{}+(0,0)*{\star}+(0,-3)*{\scs {#1}};
}}

 \newcommandx{\thickCIRCLELEFTweb}[1][1 = \lambda-1]{\xy 0;/r.18pc/:
 (0,0)*{\ticbub{#1}{}};
 \endxy}

 \newcommandx{\thickCIRCLERIGHTweb}[1][1 = \lambda-1]{\xy 0;/r.18pc/: 
 (0,0)*{\ticcbub{#1}{}};
 \endxy}

\newcommand{\crossingEEfoam}[1][1]{
\xy
(0,0)*{
\begin{tikzpicture} [fill opacity=0.2,  decoration={markings,
                        mark=at position 0.6 with {\arrow{>}};    }, scale=#1]
     	\filldraw [fill=red] (2,2) rectangle (-2,-1);
     	\draw[very thick](-2,2)--(2,2);
      	\draw[ultra thick, red] (-.75,2) arc (180:360:.75);
	\draw[very thick] (-2,-1)--(2,-1);
     	\draw[ultra thick, red] (-.75,-1) arc (180:0:.75);
    \draw[very thick, postaction={decorate}] (2,-1) -- (1,-1);
	\path [fill=blue] (1.75,-2) arc (0:60:0.75) -- (0.375,-0.35) arc (60:0:0.75);
      	\path [fill=blue] (0.375,-0.35) arc (60:180:0.75) -- (0.25,-2) arc (180:60:0.75);
	\draw[very thick, postaction={decorate}] (.25,-2) -- (-.75,-1);
     	\draw[very thick, postaction={decorate}] (1.75,-2) --(.75,-1);
	\filldraw[ultra thick, red] [fill=blue] (1,-1.25) -- (0,-0.25) -- (0,1.25) -- (1,0.25) -- cycle;
	\path [fill=blue] (.25,1) arc (180:240:0.75) -- (-0.375,1.35) arc (240:180:0.75);
	\path [fill=blue] (0.625,0.35) arc (240:360:0.75) -- (.75,2) arc (0:-120:0.75);
\begin{scope}
    \clip (1,-2) rectangle (3,1);
    \filldraw[fill=red] (-1,-2) rectangle (3,1);
    \draw[very thick](-1,1)--(3,1);
    \draw[very thick](-1,-2)--(3,-2);
    \draw[ultra thick, red](.25,-2) arc (180:0:.75);
    \draw[ultra thick, red](.25,1) arc (180:360:.75);
\end{scope}
\draw[very thick, postaction={decorate}] (3,-2) -- (2,-2);
\draw[ultra thick, red] (.25,-2) arc (180:90:.75);
\draw[ultra thick, red] (.25,1) arc (180:270:.75);
\path [fill=red] (.25,-2) arc (180:90:.75) -- (1,-2) -- cycle;
\path [fill=red] (.25,1) arc (180:270:.75) -- (1,1) -- cycle;
\draw[very thick](.25,-2)--(3,-2);
\draw[very thick] (0.25,1)--(3,1);
     	\draw[very thick, postaction={decorate}] (.25,1) -- (-.75,2);
     	\draw[very thick, postaction={decorate}] (1.75,1) --(.75,2);
    \node[black,  fill opacity= 0.5] at (2.25, 2)[below,scale=0.7] {$\lambda +  2$};
    \node[black,  fill opacity= 0.5] at (3.25,1)[below,scale=0.7] {$2$};
\end{tikzpicture}}
\endxy
}

\newcommandx{\crossingfoam}[1][1 = 1]{
\xy
(0,0)*{
\begin{tikzpicture} [fill opacity=0.2,  decoration={markings,
                        mark=at position 0.6 with {\arrow{>}};    }, scale=#1]
     	\draw[very thick](2,3)--(-4.75,3);
        \draw[very thick] (2,0)--(-4.75,0);
        \draw[](-4.75,0)--(-4.75,3);
        \draw[ultra thick, red](-2.75,0)--(-2.75,3);
        \draw[](2,0)--(2,3);
        \path[fill = red] (-4.75,3)--(-2.75,3) -- (-2.75,0)--(-4.75,0);
        \path[fill = red] (2,3)--(-2.75,3) -- (-2.75,0)--(2,0);
        \draw[very thick, postaction={decorate}] (-4,0) -- (-3,0);
	\path [fill=blue] (-2,2) -- (-2,-1) -- (-2.75,0)--(-2.75,3) -- cycle;
     	\path[fill=red] (-2,2) -- (-.75,2) arc(180:270:.75) -- (0,-.25) arc(90:180:.75) -- (-2,-1);
     	\draw[very thick](-.75,2)--(-2,2);
	\draw[very thick] (.75,2) -- (-.75,2);
	\draw (-2,2) -- (-2,-1);
      	\draw[ultra thick, red] (-.75,2) arc (180:270:.75);
	\draw[ultra thick, red] (.75,2) arc (360:270:.75);
	\path[fill=red] (.75,2) arc (360:180:.75);
	\draw[very thick] (-.75,-1)--(-2,-1);
	\draw[very thick] (.75,-1) -- (-.75,-1);
     	\draw[ultra thick, red] (-.75,-1) arc (180:90:.75);
	\draw[ultra thick, red] (.75,-1) arc (0:90:.75);
	\path[fill=red] (.75,-1) arc (0:180:.75);
	\path [fill=blue] (1.75,-2) arc (0:60:0.75) -- (0.375,-0.35) arc (60:0:0.75);
      	\path [fill=blue] (0.375,-0.35) arc (60:180:0.75) -- (0.25,-2) arc (180:60:0.75);
	\draw[very thick, postaction={decorate}] (.25,-2) -- (-.75,-1);
     	\draw[very thick, postaction={decorate}] (1.75,-2) --(.75,-1);
        \draw[very thick, postaction={decorate}](-2,-1)  --(-2.75,0);
	\path [fill=blue] (1,-1.25) -- (0,-0.25) -- (0,1.25) -- (1,0.25) -- cycle;
	\draw (0,1.25) -- (0,-.25);
	\draw (1,-1.25) -- (1,0.25);
	\draw[ultra thick, red] (0,1.25) -- (1,.25);
	\draw[ultra thick, red] (1,-1.25) -- (0,-.25);
	\path [fill=blue] (.25,1) arc (180:240:0.75) -- (-0.375,1.35) arc (240:180:0.75);
	\path [fill=blue] (0.625,0.35) arc (240:360:0.75) -- (.75,2) arc (0:-120:0.75);
 	\path[fill=red] (3,1) -- (1.75,1) arc(360:270:.75) -- (1,-1.25) arc(90:0:.75) -- (3,-2);
     	\draw[very thick] (3,1) -- (1.75,1);
	\draw[very thick] (1.75,1) -- (.25,1);
	\draw (3,1) -- (3,-2);
	\draw[ultra thick, red] (1.75,1) arc (360:270:.75);
	\draw[ultra thick, red] (.25,1) arc (180:270:.75);
	\path[fill=red] (.25,1) arc (180:360:.75);
	\draw[very thick] (3,-2) -- (1.75,-2);
     	\draw[very thick] (1.75,-2)--(.25,-2);
     	\draw[ultra thick, red] (1.75,-2) arc (0:90:.75);
	\draw[ultra thick, red] (.25,-2) arc (180:90:.75);
	\path[fill=red] (.25,-2) arc (180:0:.75);
    \draw[very thick, postaction={decorate}] (3,-2) -- (2,-2);
     	\draw[very thick, postaction={decorate}] (.25,1) -- (-.75,2);
     	\draw[very thick, postaction={decorate}] (1.75,1) --(.75,2);
                \draw[very thick, postaction={decorate}] (-2,2)  --(-2.75,3);
    \node[black,  fill opacity= 0.5] at (2.25, 3)[below,scale=0.7] {$\lambda +  2$};
    \node[black,  fill opacity= 0.5] at (3.25,2)[below,scale=0.7] {$2$};
\end{tikzpicture}}
\endxy
}

\newcommandx{\cupEFfoamTUBE}[6][1 = 1, 2 = 2, 3 = 3, 4 = 4, 5 = 5, 6 = 6]{
\xy
(0,0)*{
\begin{tikzpicture} [fill opacity=0.2,  decoration={markings,
                        mark=at position 0.6 with {\arrow{>}};    }, scale=#1]
     	\filldraw [fill=red] (-3,0) rectangle (2,3);
     	\draw[very thick](-3,3)--(2,3);
     	\draw[very thick] (-3,0)--(2,0);
      	\draw[ultra thick, red] (-.75,3) arc (180:360:.75);
	\path [fill=blue] (.25,2) arc (180:240:0.75) -- (-0.375,2.35) arc (240:180:0.75);
	\path [fill=blue] (0.625,1.35) arc (240:360:0.75) -- (.75,3) arc (0:-120:0.75);
     	\filldraw [fill=red] (-2,-1) rectangle (3,2);
     	\draw[very thick] (-2,-1)--(3,-1);
     	\draw[very thick] (-2,2)--(3,2);
     	\draw[very thick] (.25,2) -- (-.75,3) ;
     	\draw[very thick] (.75,3) -- (1.75,2);
       	\draw[ultra thick, red] (.25,2) arc (180:360:.75);
        \checkarg{#2}%
\ifshownode
    \node[black, fill opacity=0.8] at (2.75, -0.75) {$\bullet$};
    \node[black, fill opacity=0.8] at (2.75,-0.75) [above, scale=0.7] {#2}; 
\fi
        \checkarg{#3}%
\ifshownode
    \node[black, fill opacity= 0.8 ] at(-0.80, 0.80) {$\bullet$};
    \node[black,  fill opacity= 0.8] at (-0.875, 0.80)[below,scale=0.7] {#3};
    \fi
        \checkarg{#4}%
\ifshownode
    \node[black, fill opacity= 0.8 ] at (0.625, 1.625) {$\bullet$};
    \node[black,  fill opacity= 0.8] at (0.625, 1.625)[below,scale=0.7] {#4};
    \fi
            \checkarg{#5}%
\ifshownode
    \node[black, fill opacity= 0.8 ] at (0.875, 2.5) {$\bullet$};
    \node[black,  fill opacity= 0.8] at (0.875, 2.5)[below,scale=0.7] {#5};
    \fi
            \checkarg{#6}%
\ifshownode
     \node[black, fill opacity= 0.8 ] at (0.25, 2.875) {$\bullet$};
    \node[black,  fill opacity= 0.8] at (0.25, 2.95)[above,scale=0.7] {#6};
    \fi
    \node[black,  fill opacity= 0.5] at (2.25, 3)[below,scale=0.7] {$\lambda + a$};
    \node[black,  fill opacity= 0.5] at (3.25,2)[below,scale=0.7] {$a$};
    \path [fill=green] (0.75,0) arc (0:60:0.75) -- (-0.625,1.65) arc (60:0:0.75);
      	\path [fill=green] (-0.625,1.65) arc (60:180:0.75) -- (-0.75,0) arc (180:60:0.75);
    \path [fill=green] (-0.75,0) arc (180:240:0.75) -- (-1.375,0.35) arc (240:180:0.75);
	\path [fill=green] (-.375,-.65) arc (240:360:0.75) -- (-0.25,1) arc (0:-120:0.75);
    \draw[ultra thick, red, fill opacity = 0.5] (0.75,0) arc (0:360:.75);
    \draw[ultra thick, red] (-0.25,1) arc (0:360:.75);
\end{tikzpicture}}
\endxy
}

\newcommandx{\cupEFfoam}[6][1 = 1, 2 = 2, 3 = 3, 4 = 4, 5 = 5, 6 = 6]{
\xy
(0,0)*{
\begin{tikzpicture} [fill opacity=0.2,  decoration={markings,
                        mark=at position 0.6 with {\arrow{>}};    }, scale=#1]
     	\filldraw [fill=red] (-2,0) rectangle (2,3);
     	\draw[very thick](-2,3)--(2,3);
     	\draw[very thick] (-2,0)--(2,0);
      	\draw[ultra thick, red] (-.75,3) arc (180:360:.75);
	\path [fill=blue] (.25,2) arc (180:240:0.75) -- (-0.375,2.35) arc (240:180:0.75);
	\path [fill=blue] (0.625,1.35) arc (240:360:0.75) -- (.75,3) arc (0:-120:0.75);
     	\filldraw [fill=red] (-1,-1) rectangle (3,2);
     	\draw[very thick] (-1,-1)--(3,-1);
     	\draw[very thick] (-1,2)--(3,2);
     	\draw[very thick] (.25,2) -- (-.75,3) ;
     	\draw[very thick] (.75,3) -- (1.75,2);
       	\draw[ultra thick, red] (.25,2) arc (180:360:.75);
        \checkarg{#2}%
\ifshownode
    \node[black, fill opacity=0.8] at (2.75, -0.75) {$\bullet$};
    \node[black, fill opacity=0.8] at (2.75,-0.75) [above, scale=0.7] {#2}; 
\fi
        \checkarg{#3}%
\ifshownode
    \node[black, fill opacity= 0.8 ] at (-1.5, 2.5) {$\bullet$};
    \node[black,  fill opacity= 0.8] at (-1.5, 2.5)[below,scale=0.7] {#3};
    \fi
        \checkarg{#4}%
\ifshownode
    \node[black, fill opacity= 0.8 ] at (0.625, 1.625) {$\bullet$};
    \node[black,  fill opacity= 0.8] at (0.625, 1.625)[below,scale=0.7] {#4};
    \fi
            \checkarg{#5}%
\ifshownode
    \node[black, fill opacity= 0.8 ] at (0.875, 2.5) {$\bullet$};
    \node[black,  fill opacity= 0.8] at (0.875, 2.5)[below,scale=0.7] {#5};
    \fi
            \checkarg{#6}%
\ifshownode
     \node[black, fill opacity= 0.8 ] at (0.25, 2.875) {$\bullet$};
    \node[black,  fill opacity= 0.8] at (0.25, 2.95)[above,scale=0.7] {#6};
    \fi
    \node[black,  fill opacity= 0.5] at (2.25, 3)[below,scale=0.7] {$\lambda + a$};
    \node[black,  fill opacity= 0.5] at (3.25,2)[below,scale=0.7] {$a$};
\end{tikzpicture}}
\endxy
}

\newcommandx{\cupEFfoama}[4][1 = 1, 2 = 0, 3 = 0, 4 = 0]{
\xy
(0,0)*{
\begin{tikzpicture} [fill opacity=0.2,  decoration={markings,
                        mark=at position 0.6 with {\arrow{>}};    }, scale=#1]
     	\filldraw [fill=red] (-2,0) rectangle (2,3);
     	\draw[very thick](-2,3)--(2,3);
     	\draw[very thick] (-2,0)--(2,0);
      	\draw[ultra thick, red] (-.75,3) arc (180:360:.75);
        \draw[very thick, postaction={decorate}] (2,0) -- (1,0) ;
	\path [fill=blue] (.25,2) arc (180:240:0.75) -- (-0.375,2.35) arc (240:180:0.75);
	\path [fill=blue] (0.625,1.35) arc (240:360:0.75) -- (.75,3) arc (0:-120:0.75);
     	\path[fill=red] (.25,2) arc (180:360:.75);
        \draw[ultra thick, red] (.25,2) arc (180:360:.75);
        \draw[very thick](.25,2)--(1.75,2);
        \draw[very thick] (1,2) -- (0.25,2) ;
     	\draw[very thick, postaction={decorate}] (.25,2) -- (-.75,3) ;
     	\draw[very thick, postaction={decorate}] (.75,3) -- (1.75,2);
        \checkarg{#2}%
\ifshownode
    \node[black, fill opacity= 0.8 ] at (1, 1.75) {$\bullet$};
    \node[black,  fill opacity= 0.8] at (1, 1.75)[below,scale=0.7] {#2};
    \fi
        \checkarg{#3}%
\ifshownode
    \node[black, fill opacity= 0.8 ] at (-1.5, 2.5) {$\bullet$};
    \node[black,  fill opacity= 0.8] at (-1.5, 2.5)[below,scale=0.7] {#3};
    \fi
            \checkarg{#4}%
\ifshownode
     \node[black, fill opacity= 0.8 ] at (0.25, 2.875) {$\bullet$};
    \node[black,  fill opacity= 0.8] at (0.25, 2.875)[below,scale=0.7] {#4};
    \fi
    \node[black,  fill opacity= 0.5] at (2.25, 3)[below,scale=0.7] {$\lambda$};
\end{tikzpicture}}
\endxy
}

\allowdisplaybreaks

\newcommandx{\cupEFfoamb}[4][1 = 1, 2 = 0, 3 = 0, 4 = 0]{
\xy
(0,0)*{
\begin{tikzpicture} [fill opacity=0.2,  decoration={markings,
                        mark=at position 0.6 with {\arrow{>}};    }, scale=#1]
     	\path[fill=red]   (-2,0) -- (-2, 3) -- (-.75,3) arc (180:360:.75) -- (2,3) -- (2, 0);
     	\draw[very thick](-2,3)--(-.75,3);
        \draw[very thick](0.75,3)--(2,3);
     	\draw[very thick] (-2,0)--(2,0);
      	\draw[ultra thick, red] (-.75,3) arc (180:360:.75);
        \draw[](-2,0) -- (-2, 3);
        \draw[](2,0) -- (2, 3);
        \draw[very thick, postaction={decorate}] (2,0) -- (1,0) ;
	\path [fill=blue] (.25,2) arc (180:240:0.75) -- (-0.375,2.35) arc (240:180:0.75);
	\path [fill=blue] (0.625,1.35) arc (240:360:0.75) -- (.75,3) arc (0:-120:0.75);
     	\filldraw [fill=red] (-1,-1) rectangle (3,2);
     	\draw[very thick] (-1,-1)--(3,-1);
     	\draw[very thick] (-1,2)--(3,2);
        \draw[very thick, postaction={decorate}] (3,-1) -- (2,-1) ;
     	\draw[very thick, postaction={decorate}] (.25,2) -- (-.75,3) ;
     	\draw[very thick, postaction={decorate}] (.75,3) -- (1.75,2);
       	\draw[ultra thick, red] (.25,2) arc (180:360:.75);
    \checkarg{#2}%
\ifshownode
    \node[black, fill opacity=0.8] at (2.75, -0.75) {$\bullet$};
    \node[black, fill opacity=0.8] at (2.75,-0.75) [above, scale=0.7] {#2}; 
\fi
        \checkarg{#3}%
\ifshownode
    \node[black, fill opacity= 0.8 ] at (-1.5, 2.5) {$\bullet$};
    \node[black,  fill opacity= 0.8] at (-1.5, 2.5)[below,scale=0.7] {#3};
    \fi
        \checkarg{#4}%
\ifshownode
    \node[black, fill opacity= 0.8 ] at (1, 1.75) {$\bullet$};
    \node[black,  fill opacity= 0.8] at (1, 1.75)[below,scale=0.7] {#4};
    \fi
    \node[black,  fill opacity= 0.5] at (2.25, 3)[below,scale=0.7] {$1$};
    \node[black,  fill opacity= 0.5] at (3.25,2)[below,scale=0.7] {$-\lambda +  1$};
\end{tikzpicture}}
\endxy
}

\newcommandx{\tripleuparrow}[3][1 = 0, 2 = 0, 3 = 0]{
\xy
(0,0)*{
\begin{tikzpicture}
    \draw[thick, ->] (0,0) to (0,1.6);
    \draw[thick, ->] (.8,0) to (.8,1.6);
    \draw[thick, ->] (1.6,0) to (1.6,1.6);
    \node at (-.2,.1) {\tiny $i$};
    \node at (.6,.1) {\tiny $j$};
    \node at (1.4,.1) {\tiny $i$};
    \checkarg{#1}\ifshownode
    \node at (0,.8){$\bullet$};
    \node at (0,.8)[left] {\tiny #1};
    \fi
    \checkarg{#2}\ifshownode
    \node at (0.8,.8){$\bullet$};
    \node at (0.8,.8)[left] {\tiny #2};
    \fi
    \checkarg{#3}\ifshownode
    \node at (1.6,.8){$\bullet$};
    \node at (1.6,.8)[left] {\tiny #3};
    \fi
\end{tikzpicture}}
\endxy}

\begin{document}

\title{Action of the Witt algebra on categorified quantum groups}
 
\author[J. Grlj]{Jernej Grlj}
\address{Department of Mathematics\\
 University of Southern California \\
  Los Angeles, California 90089, USA}
  \email{grlj@usc.edu}
\author[A.D. Lauda]{Aaron D. Lauda}
\address{Department of Mathematics\\
 University of Southern California \\
  Los Angeles, California 90089, USA}
  \email{lauda@usc.edu}
\date{\today}

\maketitle

\begin{abstract}
We construct an action of the positive Witt algebra on the categorified quantum group associated to a simply-laced Lie algebra.  In the type A case, we show that this action induces an action of the positive Witt algebra on $\mathfrak{gl}_n$-foams, recovering the action of Qi, Robert, Sussan, and Wagner.    We also show that this construction is compatible with the trace decategorification, inducing the action of the positive Witt algebra on the current algebra.   
\end{abstract}

\section{Introduction}

Triply graded link homology theories, such as the HOMFLY-PT homology~\cite{Khovanov_2008}, often exhibit additional algebraic structure not evident in their original definitions.   One example is the action of the positive half of the Witt algebra $\witt$ by grading-shifting derivations, as constructed by Khovanov and Rozansky~\cite{khovanov2013positivehalfwittalgebra}.  This Lie algebra is spanned by vector fields $L_m = x^{m+1}\frac{d}{dx}$ acting by derivations of the polynomial algebra $\Q[x]$.   
Whenever link homology exhibits additional algebraic structure, such as differentials, filtrations, or Lie algebra actions, it is reasonable to expect that these reflect deeper features of the underlying categorified representation theory that govern link homology. Our results confirm this expectation for the case of the Witt algebra and clarify how this symmetry is embedded in the 2-categorical structure of the categorified quantum group $\dot{\mathcal{U}}_Q(\mathfrak{g})$ introduced in~\cite{Lau1,KL3,Rou2}.
 
The observation by Khovanov and Rozansky that the positive Witt algebra acts on triply graded link homology was foreshadowed by work of Gorsky, Oblomkov, and Rasmussen~\cite{Gorsky2013}, who showed that the stable Khovanov homology of torus knots admits an action of the Virasoro algebra. The Witt algebra arises as the semiclassical limit of the Virasoro algebra, and its appearance in this context reflects an underlying structure of derivations acting on link homology. Such structures also emerge naturally in the geometric model for HOMFLY-PT homology via coherent sheaves on Hilbert schemes of points in the plane~\cite{Oblomkov_2018}. The equivariant $K$-theory of these Hilbert schemes carries an action of the elliptic Hall algebra~\cite{Schiffmann_2013}, which contains Heisenberg and $\mathcal{W}$-type subalgebras. Together, these developments suggest that derivation operators in link homology reflect deeper algebraic and geometric symmetries and motivate the search for representation-theoretic explanations of such actions.

More recently, Qi, Robert, Sussan, and Wagner~\cite{qi2023symmetriesequivariant} defined an $\mathfrak{sl}_2$ action by derivations on equivariant $\mathfrak{gl}_n$ homology.  This work was then combined with  Khovanov and Rozansky~\cite{khovanov2013positivehalfwittalgebra} positive Witt algebra action by Guérin and Roz, who constructed an action of $\ourwitt$ on the $\mathfrak{gl}_n$ Khovanov-Rozansky homology~\cite{guérin2025actionwittalgebrakhovanovrozansky},   extending observations from~\cite{roz2023mathfraksl2actionlinkhomology}.   This action is slightly larger than the action observed by Khovanov and Rozansky as it includes operators $L_m$ for $m=-1$ that are part of the $\mathfrak{sl}_2$ action from~\cite{qi2023symmetriesequivariant}.  There is an alternative action on Khovanov-Rozansky homology presented in \cite{Gorsky_2014}, which also specializes to an action of $\mathfrak{sl}_2$. These homologies are often realized combinatorially using the 2-category of foams (see for example \cite{Mackaay_2009,QR,Robert2020-xp}). Thus, understanding symmetries on link homology is closely tied to understanding actions on foams.

The link between these topological invariants and the representation theory of $\mathcal{U}_{Q}(\mathfrak{g})$ is provided by categorical Skew Howe duality \cite{CKM,CautisKam-coherent,Lauda2015,QR}, which relates foams to categorified quantum groups. In this article, we construct an action of the positive Witt algebra $\ourwitt$ on the categorified quantum group $\dot{\mathcal{U}}_Q(\mathfrak{g})$ associated to a simply-laced Lie algebra. In type $A$, our construction shows that the derivation operators observed in link homology arise naturally from higher morphisms in the categorified representation theory of $\mathfrak{gl}_n$, providing a representation-theoretic explanation for the appearance of Witt algebra symmetries via skew Howe duality. While skew Howe duality specifically motivates the Witt action in type $A$, our construction extends this action uniformly to all simply-laced types $\mathfrak{g}$, suggesting that the underlying derivation structure is inherent to categorified quantum groups. Some actions of $\mathfrak{sl}_2$ on KLR algebras were considered in earlier work by Khovanov~\cite{KhovNote}.

The Witt action on $\Ucat$ is equivariant with respect to the action on foams and categorical Skew Howe duality. We show that this action recovers the $\mathfrak{sl}_2$ action on $\Ucat(\mathfrak{sl}_2)$ of~\cite{Elias_2023}, answering a question of~\cite{qi2024symmetriesmathfrakglnfoams}. Moreover, we are able to show that our action is equivariant with respect to the action on current algebras and trace decategorification of Beliakova, Habiro, Webster and one of the authors in~\cite{Lau-trace3}, see also~\cite{Beliakova_2016,Lau-trace2}.

\subsection*{Acknowledgments}
We would like to thank M. Khovanov for suggesting this problem, J. Sussan for suggesting this problem and providing comments on a preliminary version of this paper and F. Roz for providing comments on a preliminary version of this paper. We would like to thank the anonymous referee for providing a short proof of Lemma~\ref{actiononfunctions}. A.D.L. and J.G. are partially supported by NSF grant DMS-2200419 and the Simons Foundation collaboration grant on New Structures in Low-Dimensional Topology.

\section{Preliminaries}

Fix a base field $\Bbbk$. We will always work over this field, which is not assumed to be of characteristic 0, nor algebraically closed. We assume that $2$ is invertible. 

\subsection{Symmetric functions} \label{subsec_sympoly}

Let us denote by $S_k$ the symmetric group and $\Sym_k=\Z[x_1, \dots,x_k]^{S_k}$
the ring of symmetric polynomials.

Let $\Sym$ be the ring of symmetric functions, defined as a subring of
the inverse limit of the system $(\Sym_k)_{k\ge0}$ (see
e.g.~\cite{McD}).  The elementary symmetric functions
$$
e_j:=\sum_{1\leq i_1<i_2<\dots<i_j} x_{i_1}\dots x_{i_j}\,,
$$
or the complete symmetric functions
$$
h_j:=\sum_{1\leq i_1\leq i_2\leq\dots\leq i_j} x_{i_1}\dots x_{i_j}
$$
generate $\Sym$ as a free commutative ring, i.e.,
$\Sym=\Z[\elem_1,\elem_2,\ldots]=\Z[h_1,h_2,\ldots]$.
The power sum symmetric functions are defined by
$$
p_t:=\sum_i x_i^t\,.
$$
We have the Newton identities:
\begin{gather} \label{eq:Newton}
  j\elem_j=\sum_{i=1}^j(-1)^{i-1}\elem_{j-i}p_i,  \qquad 
  jh_j=\sum_{i=1}^jh_{j-i}p_i\, .
\end{gather}

\subsection{Witt algebra}
We will define a Witt sequence and the positive half of a Witt algebra following~\cite[Section 3]{qi2024symmetriesmathfrakglnfoams}. 

\begin{definition}
The Lie algebra $\witt$ is generated by symbols $(\LLn)_{n \in \Z}$
subject to the relations
\begin{align}
  \label{eq:witt-relations}
 [\LLn, \LLn[m]] 
  &= (n-m) \LLn[m+n]
\end{align}
for all $n, m\in \Z$.

The Lie algebra $\mathfrak{sl}_2$ is a three dimensional Lie algebra with generators $e, h, f$ with the following relations:
\begin{align}
    [h, e] = 2e, \quad [h, f] = -2f, \quad [e, f] = h
\end{align}
\end{definition}
We will be interested in the Lie subalgebra $\ourwitt$ generated by symbols $(\LLn)_{n\in \Z_{\geq -1}}$.

\begin{lemma}[Lemma 3.6~\cite{qi2024symmetriesmathfrakglnfoams}] \label{lem:wittsl2}
There is an injective map of Lie algebras $\iota \maps \mathfrak{sl}_2 \to \witt$ given by:
  \label{lem:sl2-to-witt}
  \begin{align*}
    \iota: \quad
    e &\mapsto \LLn[-1] \\
    h &\mapsto 2\LLn[0] \\
    f &\mapsto -\LLn[1]
  \end{align*}
 whose
  image is in $\ourwitt$.  This is a maximal finite-dimensional subalgebra of $\ourwitt$.  
\end{lemma}

Lastly, we will need the definition of a Witt sequence.
\begin{definition}
  A \emph{Witt-sequence} $(\lambda_n)_{n \in \Z_{\geq -1}} \in \Bbbk^{\infty}$ is a sequence such that $\lambda_{-1}=0$ and for any  $m,n \in \Z_{\geq 0}$,
  \begin{equation*}
n\lambda_{n} - m\lambda_{m}  = (n-m)\lambda_{m+n}.
\end{equation*}
\end{definition}

\begin{lemma}\label{actiononpoly}
The Lie algebra $\ourwitt$ acts on the ring of polynomials in infinitely many variables $\Bbbk[z_1, z_2, \dots]$ where the action is given as follows. Let $p \in \Bbbk[z_1, z_2, \dots]$, then
\begin{align*}
  \LLn \cdot p = -\sum_{i}
  z_i^{n+1}\frac{\partial p}{\partial z_i}.
\end{align*}
In particular, we have an action of $\ourwitt$ on the ring of symmetric functions $\Sym$.
\end{lemma}

\begin{lemma}\label{actiononfunctions}
    Let $h_m$, $e_m$ and $p_m$ be the symmetric functions of degree $m$ as defined in \ref{subsec_sympoly}. Then \begin{align*}
        \LLn[n](h_m) &=  -(n+m)h_{n+m} +\sum_{j = 0}^{n-1} p_{n - j} h_{m+j} \\
        \LLn[n](e_m) &=  (-1)^{n+1} (n+m)e_{n+m} + \sum_{j = 0}^{n-1} (-1)^{j+1} p_{n - j} e_{m+j}\\
        \LLn[n](p_m) &= - m p_{n + m}  
    \end{align*}
\end{lemma}
\begin{proof}
    We will only prove the first equality; the second equality follows similarly, and the last equality is evident.
    For the first equality, note that by \eqref{actiononfunctions} it suffices to show that $\LLn(h_m) = - \sum_{j = 1}^m p_{j+n} h_{m-j}$. We use the  generating function for complete symmetric functions $\sum_i h_i t^i = H(t) = \prod_{i} \frac{1}{1-x_i t}$. It follows readily that $\LLn[n]( \frac{1}{1-x_i t}) = - \frac{x_i^{n+1}t}{(1 - x_i t)^2}$, therefore
    \begin{align*}
        \LLn[n](H(t)) &= \sum_{i \geq 1} \LLn[n](h_i)t^i \\
        &= -\sum_{i \geq 1}\left(\frac{x_i^{n+1}t}{(1 - x_i t)^2} \cdot \prod_{i \neq j} (1 - x_jt)^{-1} \right) \\
        &= -\sum_{i \geq 1} \left(\frac{x_i^{n+1}t}{(1 - x_i t)} \cdot H(t) \right) \\
        &= -\sum_{i \geq 1} \sum_{k \geq 0} \left(x_i^{n+1+k}t^{k+1} \cdot H(t) \right) \\
        &= - \sum_{k \geq 1} \left(p_{k+n}t^{k} \cdot H(t) \right) \\
    \end{align*}
This completes the proof after we expand $H(t)$.
\end{proof}

%
\subsection{The 2-category $\Ucat(\mathfrak{g})$}  
%

For this article, we restrict our attention to simply-laced Kac-Moody algebras. These algebras are associated to a symmetric Cartan data consisting of
\begin{itemize}
\item a free $\Z$-module $X$ (the weight lattice),
\item for $i \in I$ ($I$ is an indexing set) there are elements $\alpha_i \in X$ (simple roots) and $\Lambda_i \in X$ (fundamental weights),
\item for $i \in I$ an element $h_i \in X^\vee = \Hom_{\Z}(X,\Z)$ (simple coroots),
\item a bilinear form $(\cdot,\cdot )$ on $X$.
\end{itemize}
Write $\langle \cdot, \cdot \rangle \maps X^{\vee} \times X
\to \Z$ for the canonical pairing. This data should satisfy:
\begin{itemize}
\item $(\alpha_i, \alpha_i) = 2$ for any $i\in I$,
\item $(\alpha_i,\alpha_j) \in \{ 0, -1\}$  for $i,j\in I$ with $i \neq j$,
\item $\la i,\lambda\ra :=\langle h_i, \lambda \rangle =  (\alpha_i,\lambda)$
  for $i \in I$ and $\lambda \in X$,
\item $\langle h_j, \Lambda_i \rangle =\delta_{ij}$ for all $i,j \in I$.
\end{itemize}
Hence $(a_{ij})_{i,j\in I}$ is a symmetrizable generalized Cartan matrix, where $a_{ij}=\langle
h_i, \alpha_j \rangle=(\alpha_i, \alpha_j)$.  We will sometimes denote the bilinear pairing $(\alpha_i,\alpha_j)$ by $i \cdot j$ and abbreviate $\la i,\lambda\ra$ to $\lambda_i$.
We denote by $X^+ \subset X$ the dominant weights which are of the form $\sum_i \lambda_i \Lambda_i$ where $\lambda_i \ge 0$.

\begin{definition}
Associated to a symmetric Cartan datum, define a {\em choice of scalars $Q$} consisting of:
\begin{itemize}
  \item $\left\{ t_{ij}  \mid \text{ for all $i,j \in I$} \right\}$,
\end{itemize}
such that
\begin{itemize}
\item $t_{ii}=0$ for all $i \in I$ and $t_{ij} \in \Bbbk^{\times}$ for $i\neq j$,
 \item $t_{ij}=t_{ji}$ when $a_{ij}=0$.
\end{itemize}
\end{definition}

The choice of scalars $Q$ controls the form of the KLR algebra $R_Q$ that governs the upward-oriented strands.  The $2$-category
$\Ucat(\mf{g})$ is controlled by the products $v_{ij}=t_{ij}^{-1}t_{ji}$ taken
over all pairs $i,j\in I$.  When the underlying graph of the simply-laced Kac-Moody algebra $\mf{g}$ is a tree, in particular a Dynkin diagram, all choices of $Q$ lead to isomorphic KLR-algebras and these isomorphisms extend to isomorphisms of categorified quantum groups $\mathcal{U}_Q(\mf{g}) \to \mathcal{U}_{Q'}(\mf{g})$ for scalars $Q$ and $Q'$~\cite{Lau-param}.

Let $\mathcal{U}_Q(\mf{g})$ denote the non-cyclic form of the categorified quantum group from~\cite{Cautis_2014}.  Though a cyclic form of the categorified quantum group has been defined~\cite{BHLW2}, for our purposes, the non-cyclic variant is the most natural version.  By~\cite[Theorem 2.1]{BHLW2}, the cyclic and non-cyclic variants are isomorphic as $2$-categories.

\begin{remark} \label{weylactionstandard}
Given an orientation of the Dynkin diagram for $\mathfrak{g}$, there is a natural choice of scalars for $\mathcal{U}_Q(\mf{g})$ where
\[
t_{ij} =
\left\{
  \begin{array}{ll}
     -1 & \hbox{ if $i \cdot j =-1$ and $i \longrightarrow j$} \\
      1 & \hbox{ otherwise. }
  \end{array}
\right.
\]
This choice of scalars $Q$ corresponds to the natural choice of KLR algebra $R_Q$ that describes ${\rm Ext}$-algebras between perverse sheaves on the quiver variety~\cite{VV}.  In this case we have $v_{ij}=t_{ij}^{-1} t_{ji} = -1$ whenever $(\alpha_i, \alpha_j)=-1$. 
\end{remark}

 \begin{definition} \label{def:2cat}
Given a choice of scalars $Q$ as above,  the 2-category $\Ucat(\mathfrak{g} )$ is the graded additive 2-category with
\begin{itemize}
\item \label{objects2cat} \textbf{Objects} $\lambda  \in X$,
\item \textbf{1-morphisms} are formal direct sums of (shifts of) compositions of
\[
\onel, \quad \1_{\lambda+\alpha_i} \mathcal{E}_i= \1_{\lambda+\alpha_i} \mathcal{E}_i\onel, \quad \text{ and }\quad
\1_{\lambda-\alpha_i} \mathcal{F}_i= \1_{\lambda-\alpha_i} \mathcal{F}_i\onel
\]
for $1 \leq i \leq n-1$ and $\lambda \in \Z^{n}$.  We denote the grading shift by $\la 1 \ra$, so that for each 1-morphism $x$ in $\mathcal{U}$ and $t\in \Z$ we have a 1-morphism $x\la t\ra$.

\item \textbf{2-morphisms} are $\Bbbk$-vector spaces spanned by compositions of coloured, decorated tangle-like diagrams illustrated below.
\begin{align}
\hackcenter{\begin{tikzpicture}[scale=0.8]
    \draw[thick, ->] (0,0) -- (0,1.5)
        node[pos=.5, shape=coordinate](DOT){};
    \filldraw  (DOT) circle (2.5pt);
    \node at (-.85,.85) {\tiny $\lambda +\alpha_i$};
    \node at (.5,.85) {\tiny $\lambda$};
    \node at (-.2,.1) {\tiny $i$};
\end{tikzpicture}} &\maps \mathcal{E}_i\onel \to \mathcal{E}_i\onel \la i\cdot i \ra  & \quad
 &
  \hackcenter{\begin{tikzpicture}[scale=0.8]
    \draw[thick, ->] (0,0) .. controls (0,.5) and (.75,.5) .. (.75,1.0);
    \draw[thick, ->] (.75,0) .. controls (.75,.5) and (0,.5) .. (0,1.0);
    \node at (1.1,.55) {\tiny $\lambda$};
    \node at (-.2,.1) {\tiny $i$};
    \node at (.95,.1) {\tiny $j$};
\end{tikzpicture}} \;\;\maps \mathcal{E}_i\mathcal{E}_j\onel  \to \mathcal{E}_j\mathcal{E}_i\onel\la -i\cdot j \ra
  \smallskip\\
\hackcenter{\begin{tikzpicture}[scale=0.8]
    \draw[thick, <-] (.75,2) .. controls ++(0,-.75) and ++(0,-.75) .. (0,2);
    \node at (.4,1.2) {\tiny $\lambda$};
    \node at (-.2,1.9) {\tiny $i$};
\end{tikzpicture}} \;\; &\maps \onel  \to \mathcal{F}_i\mathcal{E}_i\onel\la  1 + \lambda_i  \ra   &
    &
\hackcenter{\begin{tikzpicture}[scale=0.8]
    \draw[thick, ->] (.75,2) .. controls ++(0,-.75) and ++(0,-.75) .. (0,2);
    \node at (.4,1.2) {\tiny $\lambda$};
    \node at (.95,1.9) {\tiny $i$};
\end{tikzpicture}} \;\; \maps \onel  \to\mathcal{E}_i\mathcal{F}_i\onel\la  1 -  \lambda_i  \ra   \smallskip \\
\hackcenter{\begin{tikzpicture}[scale=0.8]
    \draw[thick, ->] (.75,-2) .. controls ++(0,.75) and ++(0,.75) .. (0,-2);
    \node at (.4,-1.2) {\tiny $\lambda$};
    \node at (.95,-1.9) {\tiny $i$};
\end{tikzpicture}} \;\; & \maps \mathcal{F}_i\mathcal{E}_i\onel \to\onel\la  1 +  \lambda_i  \ra   &
    &
\hackcenter{\begin{tikzpicture}[scale=0.8]
    \draw[thick, <-] (.75,-2) .. controls ++(0,.75) and ++(0,.75) .. (0,-2);
    \node at (.4,-1.2) {\tiny $\lambda$};
    \node at (-.2,-1.9) {\tiny $i$};
\end{tikzpicture}} \;\;\maps\mathcal{E}_i\mathcal{F}_i\onel  \to\onel\langle  1 -  \lambda_i  \rangle  
\end{align}

In this $2$-category (and those throughout the paper), we read diagrams from right to left and bottom to top.  The identity 2-morphism of the 1-morphism
$\mathcal{E}_i \onel$ is represented by an upward oriented line labelled by $i$ and the identity 2-morphism of $\mathcal{F}_i \onel$ is represented by a downward such line.
\end{itemize}

From the grading shifts above in the definitions of the generating 2-morphisms, it is clear that the generators have degree given as follows: 
\begin{align}
\deg \left( \hackcenter{ \begin{tikzpicture}[scale=0.8]
    \draw[thick, ->] (0,0) -- (0,1.5)
        node[pos=.5, shape=coordinate](DOT){};
    \filldraw  (DOT) circle (2.5pt);
    \node at (-.85,.85) {\tiny $\lambda +\alpha_i$};
    \node at (.5,.85) {\tiny $\lambda$};
    \node at (-.2,.1) {\tiny $i$};
\end{tikzpicture}} \right) = 2 & & 
 &
  \deg \left( \hackcenter{\begin{tikzpicture}[scale=0.8]
    \draw[thick, ->] (0,0) .. controls (0,.5) and (.75,.5) .. (.75,1.0);
    \draw[thick, ->] (.75,0) .. controls (.75,.5) and (0,.5) .. (0,1.0);
    \node at (1.1,.55) {\tiny $\lambda$};
    \node at (-.2,.1) {\tiny $i$};
    \node at (.95,.1) {\tiny $j$};
\end{tikzpicture}} \right) =  - i\cdot j 
  \smallskip\\
\deg \left( \hackcenter{\begin{tikzpicture}[scale=0.8]
    \draw[thick, <-] (.75,2) .. controls ++(0,-.75) and ++(0,-.75) .. (0,2);
    \node at (.4,1.2) {\tiny $\lambda$};
    \node at (-.2,1.9) {\tiny $i$};
\end{tikzpicture}} \right) =   1 + \lambda_i   &&
    &
\deg \left( \hackcenter{\begin{tikzpicture}[scale=0.8]
    \draw[thick, ->] (.75,2) .. controls ++(0,-.75) and ++(0,-.75) .. (0,2);
    \node at (.4,1.2) {\tiny $\lambda$};
    \node at (.95,1.9) {\tiny $i$};
\end{tikzpicture}} \right) = 1 -  \lambda_i \\
\deg \left( \hackcenter{\begin{tikzpicture}[scale=0.8]
    \draw[thick, ->] (.75,-2) .. controls ++(0,.75) and ++(0,.75) .. (0,-2);
    \node at (.4,-1.2) {\tiny $\lambda$};
    \node at (.95,-1.9) {\tiny $i$};
\end{tikzpicture}} \right) = 1 +  \lambda_i &&
    &
\deg \left( \hackcenter{\begin{tikzpicture}[scale=0.8]
    \draw[thick, <-] (.75,-2) .. controls ++(0,.75) and ++(0,.75) .. (0,-2);
    \node at (.4,-1.2) {\tiny $\lambda$};
    \node at (-.2,-1.9) {\tiny $i$};
\end{tikzpicture}} \right)  1 -  \lambda_i  
\end{align}
For a more complicated diagram of a $2$-morphism we sum up the contributions from the generating $2$-morphisms. In particular, this implies that the identity $2$-morphisms have degree $0$. Lastly we extend the degree using the $\Bbbk$-vector space structure of $2$-morphisms.

The $2$-morphisms satisfy the following relations:
\begin{enumerate}
\item \label{item_cycbiadjoint-cyc} The $1$-morphisms $\cal{E}_i \onel$ and $\cal{F}_i \onel$ are biadjoint (up to a specified degree shift).
  \item The dot $2$-morphisms are cyclic with respect to this biadjoint structure.
\begin{equation}\label{eq_cyclic_dot-cyc}
\hackcenter{\begin{tikzpicture}[scale=0.8]
    \draw[thick, ->]  (0,.4) .. controls ++(0,.6) and ++(0,.6) .. (-.75,.4) to (-.75,-1);
    \draw[thick, <-](0,.4) to (0,-.4) .. controls ++(0,-.6) and ++(0,-.6) .. (.75,-.4) to (.75,1);
    \filldraw  (0,-.2) circle (2.5pt);
    \node at (-1,.9) { $\lambda$};
    \node at (.95,.8) {\tiny $i$};
\end{tikzpicture}}
\;\; = \;\;
\hackcenter{\begin{tikzpicture}[scale=0.8]
    \draw[thick, <-]  (0,-1) to (0,1);
    \node at (.8,-.4) { $\lambda+\alpha_i$};
    \node at (-.5,-.4) { $\lambda$};
    \filldraw  (0,.2) circle (2.5pt);
    \node at (-.2,.8) {\tiny $i$};
\end{tikzpicture}}
\;\; = \;\;
\hackcenter{\begin{tikzpicture}[scale=0.8]
    \draw[thick, ->]  (0,.4) .. controls ++(0,.6) and ++(0,.6) .. (.75,.4) to (.75,-1);
    \draw[thick, <-](0,.4) to (0,-.4) .. controls ++(0,-.6) and ++(0,-.6) .. (-.75,-.4) to (-.75,1);
    \filldraw  (0,-.2) circle (2.5pt);
    \node at (1.3,.9) { $\lambda + \alpha_i$};
    \node at (-.95,.8) {\tiny $i$};
\end{tikzpicture}}
\end{equation}

The $Q$-cyclic relations for crossings are given by
\begin{equation} \label{eq_cyclic}
\hackcenter{
\begin{tikzpicture}[scale=0.8]
    \draw[thick, <-] (0,0) .. controls (0,.5) and (.75,.5) .. (.75,1.0);
    \draw[thick, <-] (.75,0) .. controls (.75,.5) and (0,.5) .. (0,1.0);
    \node at (1.1,.65) { $\lambda$};
    \node at (-.2,.1) {\tiny $i$};
    \node at (.95,.1) {\tiny $j$};
\end{tikzpicture}}
\;\; := \;\; t_{ij}^{-1}
\hackcenter{\begin{tikzpicture}[scale=0.7]
    \draw[thick, ->] (0,0) .. controls (0,.5) and (.75,.5) .. (.75,1.0);
    \draw[thick, ->] (.75,0) .. controls (.75,.5) and (0,.5) .. (0,1.0);
    \draw[thick] (0,0) .. controls ++(0,-.4) and ++(0,-.4) .. (-.75,0) to (-.75,2);
    \draw[thick] (.75,0) .. controls ++(0,-1.2) and ++(0,-1.2) .. (-1.5,0) to (-1.55,2);
    \draw[thick, ->] (.75,1.0) .. controls ++(0,.4) and ++(0,.4) .. (1.5,1.0) to (1.5,-1);
    \draw[thick, ->] (0,1.0) .. controls ++(0,1.2) and ++(0,1.2) .. (2.25,1.0) to (2.25,-1);
    \node at (-.35,.75) {  $\lambda$};
    \node at (1.3,-.7) {\tiny $i$};
    \node at (2.05,-.7) {\tiny $j$};
    \node at (-.9,1.7) {\tiny $i$};
    \node at (-1.7,1.7) {\tiny $j$};
\end{tikzpicture}}
\quad = \quad t_{ji}^{-1}
\hackcenter{\begin{tikzpicture}[xscale=-1.0, scale=0.7]
    \draw[thick, ->] (0,0) .. controls (0,.5) and (.75,.5) .. (.75,1.0);
    \draw[thick, ->] (.75,0) .. controls (.75,.5) and (0,.5) .. (0,1.0);
    \draw[thick] (0,0) .. controls ++(0,-.4) and ++(0,-.4) .. (-.75,0) to (-.75,2);
    \draw[thick] (.75,0) .. controls ++(0,-1.2) and ++(0,-1.2) .. (-1.5,0) to (-1.55,2);
    \draw[thick, ->] (.75,1.0) .. controls ++(0,.4) and ++(0,.4) .. (1.5,1.0) to (1.5,-1);
    \draw[thick, ->] (0,1.0) .. controls ++(0,1.2) and ++(0,1.2) .. (2.25,1.0) to (2.25,-1);
    \node at (1.2,.75) {  $\lambda$};
    \node at (1.3,-.7) {\tiny $j$};
    \node at (2.05,-.7) {\tiny $i$};
    \node at (-.9,1.7) {\tiny $j$};
    \node at (-1.7,1.7) {\tiny $i$};
\end{tikzpicture}} .
\end{equation}

Sideways crossings are equivalently defined by the following identities:

\begin{equation} \label{eq_crossl-gen-cyc}
\hackcenter{
\begin{tikzpicture}[scale=0.8]
    \draw[thick, ->] (0,0) .. controls (0,.5) and (.75,.5) .. (.75,1.0);
    \draw[thick, <-] (.75,0) .. controls (.75,.5) and (0,.5) .. (0,1.0);
    \node at (1.1,.65) { $\lambda$};
    \node at (-.2,.1) {\tiny $i$};
    \node at (.95,.1) {\tiny $j$};
\end{tikzpicture}}
\;\; := \;\;
\hackcenter{\begin{tikzpicture}[scale=0.7]
    \draw[thick, ->] (0,0) .. controls (0,.5) and (.75,.5) .. (.75,1.0);
    \draw[thick, ->] (.75,-.5) to (.75,0) .. controls (.75,.5) and (0,.5) .. (0,1.0) to (0,1.5);
    \draw[thick] (0,0) .. controls ++(0,-.4) and ++(0,-.4) .. (-.75,0) to (-.75,1.5);
    \draw[thick, ->] (.75,1.0) .. controls ++(0,.4) and ++(0,.4) .. (1.5,1.0) to (1.5,-.5);
    \node at (1.85,.55) {  $\lambda$};
    \node at (1.75,-.2) {\tiny $j$};
    \node at (.55,-.2) {\tiny $i$};
    \node at (-.9,1.2) {\tiny $j$};
    \node at (.25,1.2) {\tiny $i$};
\end{tikzpicture}}
 \qquad \quad
 \hackcenter{
\begin{tikzpicture}[scale=0.8]
    \draw[thick, <-] (0,0) .. controls (0,.5) and (.75,.5) .. (.75,1.0);
    \draw[thick, ->] (.75,0) .. controls (.75,.5) and (0,.5) .. (0,1.0);
    \node at (1.1,.65) { $\lambda$};
    \node at (-.2,.1) {\tiny $i$};
    \node at (.95,.1) {\tiny $j$};
\end{tikzpicture}}
\;\; = \;\;
\hackcenter{\begin{tikzpicture}[xscale=-1.0, scale=0.7]
    \draw[thick, ->] (0,0) .. controls (0,.5) and (.75,.5) .. (.75,1.0);
    \draw[thick, ->] (.75,-.5) to (.75,0) .. controls (.75,.5) and (0,.5) .. (0,1.0) to (0,1.5);
    \draw[thick] (0,0) .. controls ++(0,-.4) and ++(0,-.4) .. (-.75,0) to (-.75,1.5);
    \draw[thick, ->] (.75,1.0) .. controls ++(0,.4) and ++(0,.4) .. (1.5,1.0) to (1.5,-.5);
    \node at (-1.1,.55) {  $\lambda$};
    \node at (1.75,-.2) {\tiny $i$};
    \node at (1,-.2) {\tiny $j$};
    \node at (-.9,1.2) {\tiny $i$};
    \node at (.25,1.2) {\tiny $j$};
\end{tikzpicture}}
\end{equation}

\item The $\cal{E}$'s (respectively $\cal{F}$'s) carry an action of the KLR algebra for a fixed choice of scalars $Q$.
The KLR algebra $R$ associated to a fixed set of parameters $Q$ is defined by finite $\Bbbk$-linear combinations of braid--like diagrams in the plane, where each strand is labeled by a vertex $i \in I$.  Strands can intersect and can carry dots, but triple intersections are not allowed.  Diagrams are considered up to planar isotopy that do not change the combinatorial type of the diagram. We recall the local relations.

\begin{enumerate}

\item The quadratic KLR relations are
\begin{equation}
\hackcenter{
\begin{tikzpicture}[scale=0.8]
    \draw[thick, ->] (0,0) .. controls ++(0,.5) and ++(0,-.4) .. (.75,.8) .. controls ++(0,.4) and ++(0,-.5) .. (0,1.6);
    \draw[thick, ->] (.75,0) .. controls ++(0,.5) and ++(0,-.4) .. (0,.8) .. controls ++(0,.4) and ++(0,-.5) .. (.75,1.6);
    \node at (1.1,1.25) { $\lambda$};
    \node at (-.2,.1) {\tiny $i$};
    \node at (.95,.1) {\tiny $j$};
\end{tikzpicture}}
 \qquad = \qquad
 \left\{
 \begin{array}{ccc}
     t_{ij}\;
     \hackcenter{
\begin{tikzpicture}[scale=0.8]
    \draw[thick, ->] (0,0) to (0,1.6);
    \draw[thick, ->] (.75,0) to (.75,1.6);
    \node at (1.1,1.25) { $\lambda$};
    \node at (-.2,.1) {\tiny $i$};
    \node at (.95,.1) {\tiny $j$};
\end{tikzpicture}}&  &  \text{if $(\alpha_i,\alpha_j)=0$ or $(\alpha_i,\alpha_j)=2$,}\\ \\
  t_{ij}
  \;      \hackcenter{
\begin{tikzpicture}[scale=0.8]
    \draw[thick, ->] (0,0) to (0,1.6);
    \draw[thick, ->] (.75,0) to (.75,1.6);
    \node at (1.1,1.25) { $\lambda$}; \filldraw  (0,.8) circle (2.75pt);
    \node at (-.2,.1) {\tiny $i$};
    \node at (.95,.1) {\tiny $j$};
\end{tikzpicture}}
  \;\; + \;\; t_{ji} \;
 \hackcenter{
\begin{tikzpicture}[scale=0.8]
    \draw[thick, ->] (0,0) to (0,1.6);
    \draw[thick, ->] (.75,0) to (.75,1.6);
    \node at (1.1,1.25) { $\lambda$}; \filldraw  (.75,.8) circle (2.75pt);
    \node at (-.2,.1) {\tiny $i$};
    \node at (.95,.1) {\tiny $j$};
\end{tikzpicture}}
 
   &  & \text{if $(\alpha_i,\alpha_j)=-1$.}
 \end{array}
 \right. \label{eq_r2_ij-gen-cyc}
\end{equation}

\item The dot sliding relations are
\begin{align} \label{eq:dotslide}
\hackcenter{\begin{tikzpicture}[scale=0.8]
    \draw[thick, ->] (0,0) .. controls ++(0,.55) and ++(0,-.5) .. (.75,1)
        node[pos=.25, shape=coordinate](DOT){};
    \draw[thick, ->] (.75,0) .. controls ++(0,.5) and ++(0,-.5) .. (0,1);
    \filldraw  (DOT) circle (2.5pt);
    \node at (-.2,.15) {\tiny $i$};
    \node at (.95,.15) {\tiny $j$};
\end{tikzpicture}}
\;\; - \;\;
\hackcenter{\begin{tikzpicture}[scale=0.8]
    \draw[thick, ->] (0,0) .. controls ++(0,.55) and ++(0,-.5) .. (.75,1)
        node[pos=.75, shape=coordinate](DOT){};
    \draw[thick, ->] (.75,0) .. controls ++(0,.5) and ++(0,-.5) .. (0,1);
    \filldraw  (DOT) circle (2.5pt);
    \node at (-.2,.15) {\tiny $i$};
    \node at (.95,.15) {\tiny $j$};
\end{tikzpicture}}
\;\; = \;\;
\hackcenter{\begin{tikzpicture}[scale=0.8]
    \draw[thick, ->] (0,0) .. controls ++(0,.55) and ++(0,-.5) .. (.75,1);
    \draw[thick, ->] (.75,0) .. controls ++(0,.5) and ++(0,-.5) .. (0,1) node[pos=.75, shape=coordinate](DOT){};
    \filldraw  (DOT) circle (2.5pt);
    \node at (-.2,.15) {\tiny $i$};
    \node at (.95,.15) {\tiny $j$};
\end{tikzpicture}}
\;\; - \;\;
\hackcenter{\begin{tikzpicture}[scale=0.8]
    \draw[thick, ->] (0,0) .. controls ++(0,.55) and ++(0,-.5) .. (.75,1);
    \draw[thick, ->] (.75,0) .. controls ++(0,.5) and ++(0,-.5) .. (0,1) node[pos=.25, shape=coordinate](DOT){};
    \filldraw  (DOT) circle (2.5pt);
    \node at (-.2,.15) {\tiny $i$};
    \node at (.95,.15) {\tiny $j$};
\end{tikzpicture}}
 \;\; = \;\;
 \delta_{i,j}
\hackcenter{\begin{tikzpicture}[scale=0.8]
    \draw[thick, ->] (0,0) to  (0,1);
    \draw[thick, ->] (.75,0)to (.75,1) ;
    \node at (-.2,.15) {\tiny $i$};
    \node at (.95,.15) {\tiny $i$};
\end{tikzpicture}}.
\end{align}

\item The cubic KLR relations are
\begin{equation} \label{eq:KLRqubic}
\hackcenter{\begin{tikzpicture}[scale=0.8]
    \draw[thick, ->] (0,0) .. controls ++(0,1) and ++(0,-1) .. (1.5,2);
    \draw[thick, ] (.75,0) .. controls ++(0,.5) and ++(0,-.5) .. (0,1);
    \draw[thick, ->] (0,1) .. controls ++(0,.5) and ++(0,-.5) .. (0.75,2);
    \draw[thick, ->] (1.5,0) .. controls ++(0,1) and ++(0,-1) .. (0,2);
    \node at (-.2,.15) {\tiny $i$};
    \node at (.95,.15) {\tiny $j$};
    \node at (1.75,.15) {\tiny $k$};
\end{tikzpicture}}
\;\;- \;\;
\hackcenter{\begin{tikzpicture}[scale=0.8]
    \draw[thick, ->] (0,0) .. controls ++(0,1) and ++(0,-1) .. (1.5,2);
    \draw[thick, ] (.75,0) .. controls ++(0,.5) and ++(0,-.5) .. (1.5,1);
    \draw[thick, ->] (1.5,1) .. controls ++(0,.5) and ++(0,-.5) .. (0.75,2);
    \draw[thick, ->] (1.5,0) .. controls ++(0,1) and ++(0,-1) .. (0,2);
    \node at (-.2,.15) {\tiny $i$};
    \node at (.95,.15) {\tiny $j$};
    \node at (1.75,.15) {\tiny $k$};
\end{tikzpicture}}
\;\; = \;\;   -(\alpha_i,\alpha_j) \; \delta_{i,k} \; t_{ij}
\hackcenter{\begin{tikzpicture}[scale=0.8]
    \draw[thick, ->] (0,0) to (0,2);
    \draw[thick, -> ] (.75,0) to (0.75,2);
    \draw[thick, ->] (1.5,0) to (1.5,2);
    \node at (-.2,.15) {\tiny $i$};
    \node at (.95,.15) {\tiny $j$};
    \node at (1.75,.15) {\tiny $i$};
\end{tikzpicture}}.
\end{equation}
\end{enumerate}

\item When $i \ne j$ one has the mixed relations  relating $\cal{E}_i \cal{F}_j$ and $\cal{F}_j \cal{E}_i$
\begin{equation}  \label{mixed_rel-cyc}
 \hackcenter{\begin{tikzpicture}[scale=0.8]
    \draw[thick,<-] (0,0) .. controls ++(0,.5) and ++(0,-.5) .. (.75,1);
    \draw[thick] (.75,0) .. controls ++(0,.5) and ++(0,-.5) .. (0,1);
    \draw[thick, ->] (0,1 ) .. controls ++(0,.5) and ++(0,-.5) .. (.75,2);
    \draw[thick] (.75,1) .. controls ++(0,.5) and ++(0,-.5) .. (0,2);
        \node at (-.2,.15) {\tiny $i$};
    \node at (.95,.15) {\tiny $j$};
\end{tikzpicture}}
\;\; = \;\; t_{ij}
\hackcenter{\begin{tikzpicture}[scale=0.8]
    \draw[thick, <-] (0,0) -- (0,2);
    \draw[thick, ->] (.75,0) -- (.75,2);
     \node at (-.2,.2) {\tiny $i$};
    \node at (.95,.2) {\tiny $j$};
\end{tikzpicture}}
\qquad \qquad
 \hackcenter{\begin{tikzpicture}[scale=0.8]
    \draw[thick] (0,0) .. controls ++(0,.5) and ++(0,-.5) .. (.75,1);
    \draw[thick, <-] (.75,0) .. controls ++(0,.5) and ++(0,-.5) .. (0,1);
    \draw[thick] (0,1 ) .. controls ++(0,.5) and ++(0,-.5) .. (.75,2);
    \draw[thick, ->] (.75,1) .. controls ++(0,.5) and ++(0,-.5) .. (0,2);
        \node at (-.2,.15) {\tiny $i$};
    \node at (.95,.15) {\tiny $j$};
\end{tikzpicture}}
\;\; = \;\; t_{ji}
\hackcenter{\begin{tikzpicture}[scale=0.8]
    \draw[thick, ->] (0,0) -- (0,2);
    \draw[thick, <-] (.75,0) -- (.75,2);
     \node at (-.2,.2) {\tiny $i$};
    \node at (.95,.2) {\tiny $j$};
\end{tikzpicture}} .
\end{equation}

\item Negative degree bubbles are zero.  That is for all $m \in \Z_{>0}$ one has
\begin{equation}\label{eq:bubblesarezero}
 \hackcenter{ \begin{tikzpicture} [scale=.8]
 \draw (-.15,.35) node { $\scs i$};
 \draw[ ]  (0,0) arc (180:360:0.5cm) [thick];
 \draw[<- ](1,0) arc (0:180:0.5cm) [thick];
\filldraw  [black] (.1,-.25) circle (2.5pt);
 \node at (-.2,-.5) {\tiny $m$};
 \node at (1.15,.8) { $\lambda  $};
\end{tikzpicture} } \;  = 0\quad \text{if $m < \lambda_i -1$}, \qquad \quad
\;
\hackcenter{ \begin{tikzpicture} [scale=.8]
 \draw (-.15,.35) node { $\scs i$};
 \draw  (0,0) arc (180:360:0.5cm) [thick];
 \draw[->](1,0) arc (0:180:0.5cm) [thick];
\filldraw  [black] (.9,-.25) circle (2.5pt);
 \node at (1,-.5) {\tiny $m$};
 \node at (1.15,.8) { $\lambda $};
\end{tikzpicture} } \;  = 0 \quad  \text{if $m < -\lambda_i -1$}.
\end{equation}
Furthermore, dotted bubbles of degree zero are scalar multiples of the identity $2$-morphisms
\begin{equation} \label{eq:degreezero}
 \hackcenter{ \begin{tikzpicture} [scale=.8]
 \draw (-.15,.35) node { $\scs i$};
 \draw  (0,0) arc (180:360:0.5cm) [thick];
 \draw[,<-](1,0) arc (0:180:0.5cm) [thick];
\filldraw  [black] (.1,-.25) circle (2.5pt);
 \node at (-.5,-.5) {\tiny $\lambda_i -1$};
 \node at (1.15,1) { $\lambda  $};
\end{tikzpicture} }
\;\; =
\Id_{\1_{\l}}
\quad \text{for $  \lambda_i \geq 1$}, \qquad \quad
\;
\hackcenter{ \begin{tikzpicture} [scale=.8]
 \draw (-.15,.35) node { $\scs i$};
 \draw  (0,0) arc (180:360:0.5cm) [thick];
 \draw[->](1,0) arc (0:180:0.5cm) [thick];
\filldraw  [black] (.9,-.25) circle (2.5pt);
 \node at (1.35,-.5) {\tiny $-\lambda_i -1$};
 \node at (1.15,1) { $\lambda $};
\end{tikzpicture} }
\;\; =
\Id_{\1_{\l}} \quad \text{if $   \lambda_i \leq -1$}.
\end{equation}
 We introduce formal symbols called \emph{fake bubbles}.  These are positive degree endomorphisms of $\onel$ that carry a formal label by a negative number of dots.

\begin{itemize}
  \item Degree zero fake bubbles are normalized by
  \begin{equation}
 \hackcenter{ \begin{tikzpicture} [scale=.8]
 \draw (-.15,.35) node { $\scs i$};
 \draw  (0,0) arc (180:360:0.5cm) [thick];
 \draw[,<-](1,0) arc (0:180:0.5cm) [thick];
\filldraw  [black] (.1,-.25) circle (2.5pt);
 \node at (-.5,-.55) {\tiny $\lambda_i -1$};
 \node at (1.15,1) { $\lambda  $};
\end{tikzpicture} } \;\; =  \Id_{\1_{\l}}
\quad \text{for $  \lambda_i < 1$}, \qquad \quad
\;
\hackcenter{ \begin{tikzpicture} [scale=.8]
 \draw (-.15,.35) node { $\scs i$};
 \draw  (0,0) arc (180:360:0.5cm) [thick];
 \draw[->](1,0) arc (0:180:0.5cm) [thick];
\filldraw  [black] (.9,-.25) circle (2.5pt);
 \node at (1.35,-.5) {\tiny $-\lambda_i -1$};
 \node at (1.15,1) { $\lambda $};
\end{tikzpicture} } \;\; =
\Id_{\1_{\l}} \quad \text{if $   \lambda_i > -1$}.
\end{equation}

\item Higher degree fake bubbles for $\lambda_i <0$ are defined inductively as
\begin{equation}\label{eq:infgrassmanian1neg}
  \hackcenter{ \begin{tikzpicture} [scale=.8]
 \draw (-.15,.35) node { $\scs i$};
 \draw  (0,0) arc (180:360:0.5cm) [thick];
 \draw[,<-](1,0) arc (0:180:0.5cm) [thick];
\filldraw  [black] (.1,-.25) circle (2.5pt);
 \node at (-.65,-.55) {\tiny $\lambda_i -1+j$};
 \node at (1.15,1) { $\lambda  $};
\end{tikzpicture} } \;\; = \;\;
\left\{
  \begin{array}{ll}
    -
\displaystyle \sum_{\stackrel{\scs x+y=j}{\scs y\geq 1}} \hackcenter{ \begin{tikzpicture}[scale=.8]
 \draw (-.15,.35) node { $\scs i$};
 \draw  (0,0) arc (180:360:0.5cm) [thick];
 \draw[,<-](1,0) arc (0:180:0.5cm) [thick];
\filldraw  [black] (.1,-.25) circle (2.5pt);
 \node at (-.35,-.45) {\tiny $\overset{\lambda_i-1}{+x}$};
 \node at (.85,1) { $\lambda$};
\end{tikzpicture}  \;\;
\begin{tikzpicture}[scale=.8]
 \draw (-.15,.35) node { $\scs i$};
 \draw  (0,0) arc (180:360:0.5cm) [thick];
 \draw[->](1,0) arc (0:180:0.5cm) [thick];
\filldraw  [black] (.9,-.25) circle (2.5pt);
 \node at (1.45,-.5) {\tiny $\overset{-\lambda_i-1}{+y}$};
 \node at (1.15,1.1) { $\;$};
\end{tikzpicture}  }
  & \hbox{if $0 < j < -\lambda_i+1$;} \\
    0, & \hbox{if $j<0$.}
  \end{array}
\right.
\end{equation}

\item Higher degree fake bubbles for $\lambda_i >0$ are defined inductively by
\begin{equation}\label{eq:infgrassmanian2pos}
\hackcenter{ \begin{tikzpicture} [scale=.8]
 \draw (-.15,.35) node { $\scs i$};
 \draw  (0,0) arc (180:360:0.5cm) [thick];
 \draw[->](1,0) arc (0:180:0.5cm) [thick];
\filldraw  [black] (.9,-.25) circle (2.5pt);
 \node at (1.5,-.5) {\tiny $-\lambda_i -1+j$};
 \node at (1.15,1) { $\lambda $};
\end{tikzpicture} } \;\; = \;\;
\left\{
  \begin{array}{ll}
    -
\displaystyle\sum_{\stackrel{\scs x+y=j}{\scs x\geq 1}} \hackcenter{ \begin{tikzpicture}[scale=.8]
 \draw (-.15,.35) node { $\scs i$};
 \draw  (0,0) arc (180:360:0.5cm) [thick];
 \draw[,<-](1,0) arc (0:180:0.5cm) [thick];
\filldraw  [black] (.1,-.25) circle (2.5pt);
 \node at (-.35,-.45) {\tiny $\overset{\lambda_i-1}{+x}$};
 \node at (.85,1) { $\lambda$};
\end{tikzpicture}  \;\;
\begin{tikzpicture}[scale=.8]
 \draw (.3,.125) node {};
 \draw  (0,0) arc (180:360:0.5cm) [thick];
 \draw[->](1,0) arc (0:180:0.5cm) [thick];
\filldraw  [black] (.9,-.25) circle (2.5pt);
 \node at (1.45,-.5) {\tiny $\overset{-\lambda_i-1}{+y}$};
 \node at (1.15,1.1) { $\;$};
\end{tikzpicture}  }
  & \hbox{if $0 < j < \lambda_i+1$;} \\
    0, & \hbox{if $j<0$.}
  \end{array}
\right.
\end{equation}
\end{itemize}
The above relations are sometimes referred to as the \textit{infinite Grassmannian relations}.

\item \label{sl2relation} The $\mf{sl}_2$ relations (which we also refer to as the $\sE \sF$ and $\sF \sE$ decompositions) are:
\begin{equation}
\begin{split}
 \hackcenter{\begin{tikzpicture}[scale=0.8]
    \draw[thick] (0,0) .. controls ++(0,.5) and ++(0,-.5) .. (.75,1);
    \draw[thick,<-] (.75,0) .. controls ++(0,.5) and ++(0,-.5) .. (0,1);
    \draw[thick] (0,1 ) .. controls ++(0,.5) and ++(0,-.5) .. (.75,2);
    \draw[thick, ->] (.75,1) .. controls ++(0,.5) and ++(0,-.5) .. (0,2);
        \node at (-.2,.15) {\tiny $i$};
    \node at (.95,.15) {\tiny $i$};
     \node at (1.1,1.44) { $\lambda $};
\end{tikzpicture}}
\;\; + \;\;
\hackcenter{\begin{tikzpicture}[scale=0.8]
    \draw[thick, ->] (0,0) -- (0,2);
    \draw[thick, <-] (.75,0) -- (.75,2);
     \node at (-.2,.2) {\tiny $i$};
    \node at (.95,.2) {\tiny $i$};
     \node at (1.1,1.44) { $\lambda $};
\end{tikzpicture}}
\;\; = \;\;
\sum_{\overset{f_1+f_2+f_3}{=\lambda_i-1}}\hackcenter{
 \begin{tikzpicture}[scale=0.8]
 \draw[thick,->] (0,-1.0) .. controls ++(0,.5) and ++ (0,.5) .. (.8,-1.0) node[pos=.75, shape=coordinate](DOT1){};
  \draw[thick,<-] (0,1.0) .. controls ++(0,-.5) and ++ (0,-.5) .. (.8,1.0) node[pos=.75, shape=coordinate](DOT3){};
 \draw[thick,->] (0,0) .. controls ++(0,-.45) and ++ (0,-.45) .. (.8,0)node[pos=.25, shape=coordinate](DOT2){};
 \draw[thick] (0,0) .. controls ++(0,.45) and ++ (0,.45) .. (.8,0);
 \draw (-.15,.7) node { $\scs i$};
\draw (1.05,0) node { $\scs i$};
\draw (-.15,-.7) node { $\scs i$};
 \node at (.95,.65) {\tiny $f_3$};
 \node at (-.55,-.05) {\tiny $\overset{-\lambda_i-1}{+f_2}$};
  \node at (.95,-.65) {\tiny $f_1$};
 \node at (1.65,.3) { $\lambda $};
 \filldraw[thick]  (DOT3) circle (2.5pt);
  \filldraw[thick]  (DOT2) circle (2.5pt);
  \filldraw[thick]  (DOT1) circle (2.5pt);
\end{tikzpicture} }
\\
 \hackcenter{\begin{tikzpicture}[scale=0.8]
    \draw[thick,<-] (0,0) .. controls ++(0,.5) and ++(0,-.5) .. (.75,1);
    \draw[thick] (.75,0) .. controls ++(0,.5) and ++(0,-.5) .. (0,1);
    \draw[thick, ->] (0,1 ) .. controls ++(0,.5) and ++(0,-.5) .. (.75,2);
    \draw[thick] (.75,1) .. controls ++(0,.5) and ++(0,-.5) .. (0,2);
        \node at (-.2,.15) {\tiny $i$};
    \node at (.95,.15) {\tiny $i$};
     \node at (1.1,1.44) { $\lambda $};
\end{tikzpicture}}
\;\; + \;\;
\hackcenter{\begin{tikzpicture}[scale=0.8]
    \draw[thick, <-] (0,0) -- (0,2);
    \draw[thick, ->] (.75,0) -- (.75,2);
     \node at (-.2,.2) {\tiny $i$};
    \node at (.95,.2) {\tiny $i$};
     \node at (1.1,1.44) { $\lambda $};
\end{tikzpicture}}
\;\; = \;\;
\sum_{\overset{f_1+f_2+f_3}{=-\lambda_i-1}}\hackcenter{
 \begin{tikzpicture}[scale=0.8]
 \draw[thick,<-] (0,-1.0) .. controls ++(0,.5) and ++ (0,.5) .. (.8,-1.0) node[pos=.75, shape=coordinate](DOT1){};
  \draw[thick,->] (0,1.0) .. controls ++(0,-.5) and ++ (0,-.5) .. (.8,1.0) node[pos=.75, shape=coordinate](DOT3){};
 \draw[thick ] (0,0) .. controls ++(0,-.45) and ++ (0,-.45) .. (.8,0)node[pos=.25, shape=coordinate](DOT2){};
 \draw[thick, ->] (0,0) .. controls ++(0,.45) and ++ (0,.45) .. (.8,0);
 \draw (-.15,.7) node { $\scs i$};
\draw (1.05,0) node { $\scs i$};
\draw (-.15,-.7) node { $\scs i$};
 \node at (.95,.65) {\tiny $f_3$};
 \node at (-.55,-.05) {\tiny $\overset{\lambda_i-1}{+f_2}$};
  \node at (.95,-.65) {\tiny $f_1$};
 \node at (1.65,.3) { $\lambda $};
 \filldraw[thick]  (DOT3) circle (2.5pt);
  \filldraw[thick]  (DOT2) circle (2.5pt);
  \filldraw[thick]  (DOT1) circle (2.5pt);
\end{tikzpicture} } . \label{eq:sl2}
\end{split}
\end{equation}
\end{enumerate}
\end{definition}

It is sometimes convenient to use a shorthand notation for the bubbles that emphasizes their degrees. 
\begin{equation}
\label{starnotation}
\CIRCLEweb[0.5][<][r][\lambda][black][0] \;\;:= \;\; \hackcenter{ \begin{tikzpicture} [scale=.8]
 \draw  (-.75,1) arc (360:180:.45cm) [thick];
 \draw[<-](-.75,1) arc (0:180:.45cm) [thick];
     \filldraw  [black] (-1.55,.75) circle (2.5pt);
        \node at (-1.3,.3) { $\scriptstyle \ast + r$};
        \node at (-1.4,1.7) { $i $};
 \node at (-.2,1.5) { $\lambda $};
\end{tikzpicture}}
\;\; := \;\;
\hackcenter{ \begin{tikzpicture} [scale=.8]
 \draw  (-.75,1) arc (360:180:.45cm) [thick];
 \draw[<-](-.75,1) arc (0:180:.45cm) [thick];
     \filldraw  [black] (-1.55,.75) circle (2.5pt);
        \node at (-1.3,.3) { $\scriptstyle \lambda_i -1 + r$};
        \node at (-1.4,1.7) { $i $};
 \node at (-.2,1.5) { $\lambda $};
\end{tikzpicture}}
\qquad \qquad
\CIRCLEweb[0.5][>][r][\lambda][black][0] \;\; := \;\; \hackcenter{\begin{tikzpicture} [scale=.8]
 \draw  (-.75,1) arc (360:180:.45cm) [thick];
 \draw[->](-.75,1) arc (0:180:.45cm) [thick];
     \filldraw  [black] (-1.55,.75) circle (2.5pt);
        \node at (-1.3,.3) { $\scriptstyle \ast + r$};
        \node at (-1.4,1.7) { $i $};
 \node at (-.2,1.5) { $\lambda $};
\end{tikzpicture}}
\;\; := \;\;
\hackcenter{ \begin{tikzpicture} [scale=.8]
 \draw  (-.75,1) arc (360:180:.45cm) [thick];
 \draw[->](-.75,1) arc (0:180:.45cm) [thick];
     \filldraw  [black] (-1.55,.75) circle (2.5pt);
        \node at (-1.3,.3) { $\scriptstyle -\lambda_i -1 + r$};
        \node at (-1.4,1.7) { $i $};
 \node at (-.2,1.5) { $\lambda $};
\end{tikzpicture}}
\end{equation}

\begin{definition}
A $2$-representation of $\Ucat(\mf{g})$ is a graded additive $\Bbbk$-linear $2$-functor $\Ucat(\mf{g}) \to \mathcal{K}$ for some graded, additive $2$-category $\mathcal{K}$.
\end{definition}

\subsection{Symmetric functions} \label{subsec:symmetric functions}

The higher relations in the 2-category $\Ucat$ give rise to some surprising connections between the graded endomorphism algebra $\END_{\Ucat}(\onel)$ and the algebra of symmetric functions. 
Recall from~\cite[Corollary 4.8]{Elias_2016} that for each $i\in I$, there is an injective map from the ring of symmetric functions into $\END_{\Ucat}(\onel)$ given for $\lambda \geq 0$ by 

\begin{eqnarray}
 \phi^n \maps \Sym = \Z[e_1,e_2, \dots] &\to& \END_{\Ucat}(\onel),
 \\ \label{eq_isom_h}
  e_r(\underline{x}) & \mapsto & (-1)^r \; 
    \hackcenter{ \begin{tikzpicture} [scale=.8]
 \draw  (-.75,1) arc (360:180:.45cm) [thick];
 \draw[->](-.75,1) arc (0:180:.45cm) [thick];
     \filldraw  [black] (-1.55,.75) circle (2.5pt);
        \node at (-1.3,.3) { $\scriptstyle \ast + r$};
        \node at (-1.4,1.7) { $i $};
 \node at (-.2,1.5) { $\lambda $};
\end{tikzpicture}}
\end{eqnarray}
where $\underline{x}$ is an infinite set of variables and $e_r(\underline{x})$ is the elementary symmetric function.  This isomorphism can alternatively be described using fake bubbles
\begin{eqnarray}
 \phi^n \maps \Sym &\to& \END_{\Ucat}(\onel),
 \\ \label{eq_isom_e}
 h_r(\underline{x}) & \mapsto & 
 \hackcenter{ \begin{tikzpicture} [scale=.8]
 \draw  (-.75,1) arc (360:180:.45cm) [thick];
 \draw[<-](-.75,1) arc (0:180:.45cm) [thick];
     \filldraw  [black] (-1.55,.75) circle (2.5pt);
        \node at (-1.3,.3) { $\scriptstyle \ast + r$};
        \node at (-1.4,1.7) { $i $};
 \node at (-.2,1.5) { $\lambda $};
\end{tikzpicture}}
\end{eqnarray}

In what follows, it will be interesting to consider which products of closed diagrams correspond to the $\Q$-basis of $\Sym$ given by the power sum $p_r$ symmetric functions (see e.g. p.16 in~\cite{McD}).
Using a formula that expresses power sum symmetric functions in terms of products of complete and elementary symmetric functions, we can denote by $p_{i,r}(\lambda)$ for $r>0$, the image of the power sum symmetric polynomial on $i$-labelled strands:
\begin{equation} \label{eq_defpil}
 p_{i,r}(\lambda) := \sum_{a+b=r} (a+1)
  \hackcenter{ \begin{tikzpicture}[scale=.8]
 \draw (-.15,.35) node { $\scs i$};
 \draw  (0,0) arc (180:360:0.5cm) [thick];
 \draw[,<-](1,0) arc (0:180:0.5cm) [thick];
\filldraw  [black] (.1,-.25) circle (2.5pt);
 \node at (-.35,-.45) {\tiny $\overset{\lambda_i-1}{+a}$};
 \node at (.85,1) { $\lambda$};
\end{tikzpicture}  \;\;
\begin{tikzpicture}[scale=.8]
 \draw (-.15,.35) node { $\scs i$};
 \draw  (0,0) arc (180:360:0.5cm) [thick];
 \draw[->](1,0) arc (0:180:0.5cm) [thick];
\filldraw  [black] (.9,-.25) circle (2.5pt);
 \node at (1.45,-.5) {\tiny $\overset{-\lambda_i-1}{+b}$};
 \node at (1.15,1.1) { $\;$};
\end{tikzpicture}  }
 =  - \sum_{a+b=r} (b+1)
  \hackcenter{ \begin{tikzpicture}[scale=.8]
 \draw (-.15,.35) node { $\scs i$};
 \draw  (0,0) arc (180:360:0.5cm) [thick];
 \draw[,<-](1,0) arc (0:180:0.5cm) [thick];
\filldraw  [black] (.1,-.25) circle (2.5pt);
 \node at (-.35,-.45) {\tiny $\overset{\lambda_i-1}{+a}$};
 \node at (.85,1) { $\lambda$};
\end{tikzpicture}  \;\;
\begin{tikzpicture}[scale=.8]
 \draw (-.15,.35) node { $\scs i$};
 \draw  (0,0) arc (180:360:0.5cm) [thick];
 \draw[->](1,0) arc (0:180:0.5cm) [thick];
\filldraw  [black] (.9,-.25) circle (2.5pt);
 \node at (1.45,-.5) {\tiny $\overset{-\lambda_i-1}{+b}$};
 \node at (1.15,1.1) { $\;$};
\end{tikzpicture}  }
\end{equation}
For later convenience we set $p_{i,0}(\l) = \la i,\lambda\ra = \lambda_i$.

The bubble sliding equations imply the following power sum slide rule from~\cite{Lau-trace3}
\begin{eqnarray}
    \label{eq_powerslide2}
\hackcenter{ \begin{tikzpicture} [scale=.75]
\draw[thick, , ->] (2.5,0) to (2.5,2) ;
  \node at (1,.8) { $p_{i,r}(\lambda+\alpha_i)$};
 \node at (2.5,-.3) { $j $};
 \node at (3.25, 2.2) { $\lambda $};
\end{tikzpicture}}
 & = &
 \begin{cases}
\hackcenter{ \begin{tikzpicture} [scale=.75]
\draw[thick, , ->] (2.5,0) to (2.5,2) ;
  \node at (3.5,.8) { $p_{i,r}(\lambda)$};
 \node at (2.5,-.3) { $j $};
 \node at (3.25, 2.2) { $\lambda $};
\end{tikzpicture}}
+ 2 \;\; 
\hackcenter{ \begin{tikzpicture} [scale=.75]
\draw[thick, , ->] (2.5,0) to (2.5,2) ;
  \filldraw  (2.5,1.3) circle (2.75pt);
    \node at (2.5,-.3) { $j $};
 \node at (2.8,1.4) { $r$};
 \node at (3.25, 2.2) { $\lambda $};
\end{tikzpicture}} & \mbox{if } i = j, \\ & \\
\hackcenter{ \begin{tikzpicture} [scale=.75]
\draw[thick, , ->] (2.5,0) to (2.5,2) ;
  \node at (3.5,.8) { $p_{i,r}(\lambda)$};
 \node at (2.5,-.3) { $j $};
 \node at (3.25, 2.2) { $\lambda $};
\end{tikzpicture}}
 \;\; -  \;\; (-v_{ij})^r\;\;
\hackcenter{ \begin{tikzpicture} [scale=.75]
\draw[thick, , ->] (2.5,0) to (2.5,2) ;
  \filldraw  (2.5,1.3) circle (2.75pt);
    \node at (2.5,-.3) { $j $};
 \node at (2.8,1.4) { $r$};
 \node at (3.25, 2.2) { $\lambda $};
\end{tikzpicture}}
 & \mbox{if } a_{ij} = -1.
 \end{cases}
\end{eqnarray}

\section{Action of the Witt algebra on $\Ucat$}

\subsection{Defining action of the Witt algebra}

Throughout this section, we assume that $\mathfrak{g}$ is a simply-laced Kac-Moody algebra.

\begin{thm}\label{main}
    Let $(\mu_n)_{n \geq -1} \in \Bbbk^{\infty}$ be a Witt sequence.  There is an action of the  positive half of the Witt algebra $\ourwitt$ on the space of 2-morphisms in $\Ucat(\mathfrak{g})$ defined on generators by $\LLn[-1]$ killing all the generators except for a single strand with a dot which is mapped to an undecorated strand. For all other $n\geq 0$ we define the action as follows:
  \begin{align*}
\LLn[-1]\left(\xy
(0,0)*{ \begin{tikzpicture}[scale=0.6]
    \draw[thick, ->]  (0,-1) to (0,1);
    \node at (.8,-.4) {\scriptsize $\lambda$};
    \filldraw  (0,.2) circle (2.5pt);
    \node at (-.2,-.6) {\tiny $i$};
\end{tikzpicture}} \endxy \right) = \;\; - \xy
(0,0)*{ \begin{tikzpicture}[scale=0.6]
    \draw[thick, ->]  (0,-1) to (0,1);
    \node at (.8,-.4) {\scriptsize $\lambda$};
    \node at (-.2,-.6) {\tiny $i$};
\end{tikzpicture}} \endxy 
     &\qquad  \LLn[n]\left(\xy
(0,0)*{ \begin{tikzpicture}[scale=0.6]
    \draw[thick, ->]  (0,-1) to (0,1);
    \node at (.8,-.4) {\scriptsize $\lambda$};
    \filldraw  (0,.2) circle (2.5pt);
    \node at (-.2,-.6) {\tiny $i$};
\end{tikzpicture}} \endxy \right) = \;\; - \xy
(0,0)*{ \begin{tikzpicture}[scale=0.6]
    \draw[thick, ->]  (0,-1) to (0,1);
    \node at (.8,-.4) {\scriptsize $\lambda$};
    \filldraw  (0,.2) circle (2.5pt);
    \node at (0,.2)[right] {\scriptsize $n+1$};
    \node at (-.2,-.6) {\tiny $i$};
\end{tikzpicture}} \endxy
\qquad  \LLn[n]\left(\xy
(0,0)*{ \begin{tikzpicture}[scale=0.6]
    \draw[thick, ->]  (0,-1) to (0,1);
    \node at (.8,-.4) {\scriptsize $\lambda$};
    \node at (-.2,-.6) {\tiny $i$};
\end{tikzpicture}} \endxy \right) \; = \; 0
\end{align*}
\begin{align*}
    \LLn[n] \left(\CUPweb[1][>][0] \right) = \left(\lambda_i \left(\mu_n + \frac{1}{2}\right) - \frac{n+1}{2}\right) \CUPweb[1][>][n] - \mu_n \;\; p_{i,n}(\lambda) \CUPweb[1][>][0] + \frac{1}{2} \sum_{k + l = n, k \neq 0}  p_{i,k}(\lambda) \CUPweb[1][>][l]
\end{align*}
\begin{align*}
    \LLn[n]\left(\CUPweb[1][<][0]\right) = \left(\lambda_i \left(\mu_n - \frac{1}{2}\right)  - \frac{n+1}{2}\right) \CUPweb[1][<][n] - \mu_n \;\;  p_{i,n}(\lambda) \CUPweb[1][<][0] - \frac{1}{2} \sum_{k + l = n, k \neq 0} p_{i,k}(\lambda) \CUPweb[1][<][l]
\end{align*}
\begin{align*}
    \LLn[n]\left(\CAPweb[1][>][0]\right) = \left(-\lambda_i \left(\mu_n + \frac{1}{2}\right) - \frac{n+1}{2}\right) \CAPweb[1][>][n] + \mu_n \;\; p_{i,n}(\lambda) \CAPweb[1][>][0] - \frac{1}{2} \sum_{k + l = n, k \neq 0}   p_{i,k}(\lambda) \CAPweb[1][>][l]
\end{align*}
\begin{align*} 
    \LLn[n]\left(\CAPweb[1][<][0]\right) = \left(- \lambda_i \left(\mu_n - \frac{1}{2}\right) - \frac{n+1}{2}\right) \CAPweb[1][<][n] + \mu_n \;\; p_{i,n}(\lambda) \CAPweb[1][<][0] + \frac{1}{2} \sum_{k + l = n, k \neq 0}   p_{i,k}(\lambda) \CAPweb[1][<][l]
\end{align*}
\begin{align*}
    \LLn[n]\left(\Xwebwithtext[1][0][0][0][0][i][j]\right) = 
    \begin{cases}
        2\mu_n \left( \Xwebwithtext[1][0][$n$][0][0][$i$][$j$]  - \Xwebwithtext[1][0][0][0][$n$][$i$][$j$] \right) + 
        \displaystyle\sum_{k + l = n} \left(\Xwebwithtext[1][0][0][k][l][i][j]  \right) &\quad \text{if } i=j \\  
        2\mu_n \left( \Xwebwithtext[1][0][0][0][n][i][j] -  \Xwebwithtext[1][0][0][n][0][i][j]\right) 
        - \frac{1}{2} \displaystyle \sum_{k + l = n}  \left( \Xwebwithtext[1][0][0][$k$][$l$][$i$][$j$]  \right) &\quad \text{if } |i - j| = 1 \\
         0 &\quad \text{otherwise}
    \end{cases}
\end{align*}  
where we extend the action of $\LLn[n]$ on composition of $2$-morphisms $f \circ g$ by derivations as $\LLn[n](g\circ f ) = \LLn[n](g) \circ f + g \circ \LLn[n](f)$. 
\end{thm}

From the definition of generating 2-morphisms in Theorem~\ref{main}, the derivation property immediately implies the following Lemmas. Note that the action on the dot morphism carries a minus sign.  

\begin{lemma}\label{actiononbubbles-1}
    The action of $\LLn[-1]$ on bubbles (including fake bubbles) is given as:  
\begin{align*}
    \LLn[-1]\left( \CIRCLEweb[0.5][<][m][\lambda][black][0] \right) &= - (\lambda_i - 1 + m)\CIRCLEweb[0.5][<][m-1][\lambda][black][0] \\
     \LLn[-1]\left( \CIRCLEweb[0.5][>][m][\lambda][black][0] \right) &= - (-\lambda_i - 1 + m)\CIRCLEweb[0.5][>][m-1][\lambda][black][0]
\end{align*}
\end{lemma}

\begin{lemma}\label{actiononbubbles}
    The action of $\LLn[n]$ (for $n \geq 0$) on bubbles (including fake bubbles) respects the action of $\LLn[n]$ on symmetric functions (see \ref{actiononfunctions}) and is given by:
\begin{align*}
    \LLn[n]\left( \CIRCLEweb[0.5][<][m][\lambda][black][0] \right) &= (-n- m) \CIRCLEweb[0.5][<][m+n][\lambda][black][0] + \sum_{k + l = n, k \neq 0} \;\; p_{i,k}(\lambda) \CIRCLEweb[0.5][<][m + l][\lambda][black][0] \\
     \LLn[n]\left( \CIRCLEweb[0.5][>][m][\lambda][black][0] \right) &= (-n- m) \CIRCLEweb[0.5][>][m+n][\lambda][black][0] - \sum_{k + l = n, k \neq 0} \;\; p_{i,k}(\lambda) \CIRCLEweb[0.5][>][m + l][\lambda][black][0]
\end{align*}
\end{lemma}

\begin{lemma}
    We get the action of $\LLn[n]$ (for $n \geq -1$)  on powersum bubbles as 
\begin{align*}
    \LLn[n]\left(p_{i,m}(\lambda) \right) &= -m \quad p_{i,m+n}(\lambda)~.
\end{align*}
 
\end{lemma}
\begin{proof}
    While this could be proved by a direct computation, we show that this is a consequence of Lemma~\ref{actiononfunctions}. Since we know how $\LLn[n]$ acts on bubbles by Lemma~\ref{actiononbubbles} and that action follows the action of $\LLn[n]$ on elementary and complete homogeneous symmetric functions, we may apply the fact that those symmetric functions generate the space of symmetric functions and can be used to express the power sum function via a recursive application of Newton's identities from~\eqref{eq:Newton}.

    The case of $\LLn[-1]$ needs to be treated separately as there are some $\lambda_i$'s appearing in Lemma~\ref{actiononbubbles-1}. An argument similar to the one above for the general case works. The only difference is that we need to treat $\lambda_i$'s as a formal variable. See~\cite[Lemma 6.7]{Elias_2023} for details. Note that their action of $\LLn[-1]$, called $\mathbf{z}$ in their work, differs by a minus sign. Additionally, the case when $m = 1$ is taken care of by our definition of $p_{i,0}(\lambda) = \lambda_i$ in~\eqref{eq_defpil}.
\end{proof}

\begin{remark}
    The above proof may be modified to provide an alternative proof for the action of $\LLn[n]$ on $e_m$'s from Lemma~\ref{actiononfunctions} after we establish that the action on $h_m$'s agrees with the action on clockwise bubbles.
\end{remark}

\subsection{Proof of Theorem~\ref{main}}
We prove Theorem~\ref{main} by proving that the $\LLn$ satisfy the relations of $\ourwitt$ and give well-defined 2-functors on $\Ucat$. We split the proof into two lemmas.

\begin{lemma} \label{lem:defining}
    The action of $\LLn[i]$'s described in Theorem~\ref{main} preserves the Witt algebra defining relations, so that $[\LLn[a], \LLn[b]] D = (a- b)\LLn[a+b] D$ for any $2$-morphism $D \in \Ucat(\mathfrak{g})$. Here we do not need to restrict to scalars from Remark~\ref{weylactionstandard} as we are not using any relations in the $2$-category. 
\end{lemma}
\begin{proof}
    It suffices to verify the claim on generating 2-morphisms by direct computation. We will present the proof for one of the cups. The rest follows similarly. If $a = -1$ or $b = -1$, the computations become simpler as $\LLn[-1]$ acts trivially on the majority of the generators.
\begin{align*}
         \LLn[b]&\left(\LLn[a]\left(\CUPweb[1][>][0]\right)\right) = \LLn[b] \left( \left(\lambda_i \left(\mu_a + \frac{1}{2}\right) - \frac{n+1}{2}\right) \CUPweb[1][>][a] - \mu_a \quad p_{i,a}(\lambda) \CUPweb[1][>][0] + \right. 
         \\
         & \quad + \left. \frac{1}{2} \sum_{k + l = a, k \neq 0} \quad p_{i,k}(\lambda) \CUPweb[1][>][l] \right) = \\
         &= -a \left(\lambda_i \left(\mu_a + \frac{1}{2}\right) - \frac{n+1}{2}\right) \CUPweb[1][>][a + b] + \left(\lambda_i \left(\mu_a + \frac{1}{2}\right) - \frac{n+1}{2}\right) \left(  \left(\lambda_i \left(\mu_b + \frac{1}{2}\right) - \frac{n+1}{2}\right) \CUPweb[1][>][a+ b] - \right. 
\\
         &\quad - \mu_b \quad p_{i,b}(\lambda) \CUPweb[1][>][a]  + \left. \frac{1}{2} \sum_{k + l = b, k \neq 0} \quad p_{i,k}(\lambda) \CUPweb[1][>][l + a]  \right) + a \mu_a \quad p_{i,a + b}(\lambda) \CUPweb[1][>][0] - \\
         &\quad - \mu_a \quad p_{i,a} \left(\left(\lambda_i \left(\mu_b + \frac{1}{2}\right) - \frac{n+1}{2}\right) \CUPweb[1][>][b] - \mu_b \quad p_{i,b}(\lambda) \CUPweb[1][>][0] + \frac{1}{2} \sum_{k + l = b, k \neq 0} \quad p_{i,k}(\lambda) \CUPweb[1][>][l]  \right) - \\
        &\quad - \frac{1}{2}  \sum_{k + l = a, k \neq 0} k \quad p_{i,k + b}(\lambda) \CUPweb[1][>][l] - \frac{1}{2}  \sum_{k + l = a, k \neq 0} l \quad p_{i,k}(\lambda) \CUPweb[1][>][l + b] \\
        &\quad +  \frac{1}{2}  \sum_{k + l = a, k \neq 0} \quad p_{i,k}(\lambda) \left(  \left(\lambda_i \left(\mu_b + \frac{1}{2}\right) - \frac{n+1}{2}\right) \CUPweb[1][>][l+ b] - \mu_b \quad p_{i,b}(\lambda) \CUPweb[1][>][l] + \right. 
        \\ &\quad + \left. \frac{1}{2} \sum_{g + h = b, g \neq 0} \quad p_{i,g}(\lambda) \CUPweb[1][>][l + h]  \right)
    \end{align*}
Comparing terms in $\LLn[b]\left(\LLn[a]\left(\CUPweb[1][>][0]\right)\right)$ and $\LLn[a]\left(\LLn[b]\left(\CUPweb[1][>][0]\right)\right)$ we compute, using the fact that $(\mu_i)_{i\geq -1}$ is a Witt sequence. Without loss of generality, we assume $b \geq a$ and use that for splitting up the summations.
    \begin{align*}
      [\LLn[a]&, \LLn[b]]\left(\CUPweb[1][>][0]\right) = \LLn[b]\left(\LLn[a]\left(\CUPweb[1][>][0]\right)\right) - \LLn[a]\left(\LLn[b]\left(\CUPweb[1][>][0]\right)\right) \\
      &= -a \left(\lambda_i \left(\mu_a + \frac{1}{2}\right) - \frac{n+1}{2}\right) \CUPweb[1][>][a + b]  + b \left(\lambda_i \left(\mu_b + \frac{1}{2}\right) - \frac{n+1}{2}\right) \CUPweb[1][>][a + b] + \\
      &\quad +  a \mu_a \quad p_{i,a + b}(\lambda) \CUPweb[1][>][0]  -  b \mu_b \quad p_{i,a + b}(\lambda) \CUPweb[1][>][0]  -\\
      &\quad - \frac{1}{2}  \sum_{k + l = a, l \neq a} (a - l) \quad p_{i,a + b -l}(\lambda) \CUPweb[1][>][l] - \frac{1}{2}  \sum_{k + l = a, k \neq 0} (a- k) \quad p_{i,k}(\lambda) \CUPweb[1][>][a+ b - k] + \\
      &\quad + \frac{1}{2}  \sum_{k + l = b, l \neq b} (b - l) \quad p_{i, a + b - l}(\lambda) \CUPweb[1][>][l] + \frac{1}{2}  \sum_{k + l = b, k \neq 0} (b- k) \quad p_{i,k}(\lambda) \CUPweb[1][>][a + b - k] 
\end{align*}
Collecting like terms and simplifying, this reduces to    
\begin{align*} 
      &= (b- a) \left(\lambda_i \left(\mu_{a+b} + \frac{1}{2}\right) - \frac{n+1}{2}\right) \CUPweb[1][>][a + b] - (b-a) \mu_{a+b} \quad p_{i,a + b}(\lambda) \CUPweb[1][>][0]+ \\
      &\quad + \frac{b-a}{2} \sum_{k+l = a+b, k \neq 0}  \quad p_{i,k}(\lambda) \CUPweb[1][>][l] \\
      &= (b-a)\LLn[a+b] (\CUPweb[1][>][0])
    \end{align*}
establishing the claim. 
\end{proof}

\begin{lemma}\label{lemmabasicrelation}
The 2-functors $\LLn$ from Theorem~\ref{main} preserve the relations on the $2$-morphisms of $\Ucat(\mathfrak{g})$ for the choice of scalars from Remark~\ref{weylactionstandard}.
\end{lemma}
\begin{proof}
    Following the argument of~\cite[Lemma 4.2]{qi2024symmetriesmathfrakglnfoams}, it suffices that the action of $\LLn[n]$'s respects the relations of Definition~\ref{def:2cat}. 
%
This is verified by direct computation.  We give several examples to illustrate the key techniques.   The remaining relations are proven in a similar manner. The adjointness of cups and caps follows after a power sum bubble migration. Which gives
    \begin{align*}
    \LLn[n]\left(\CUPweb[1][>][0][\lambda + \alpha_i]\right) &= \left(\left(\lambda + \alpha_i\right)_i \left(\mu_n + \frac{1}{2}\right) + \frac{n-1}{2} - 2\mu_n \right) \CUPweb[1][>][n][\lambda + \alpha_i] - \\
    \quad & \quad - \mu_n \CUPwebwithtext[1][>][0][\lambda + \alpha_i][black][0][p_{i,k}(\lambda)] + \frac{1}{2} \sum_{k + l = n, k \neq 0}  \CUPwebwithtext[1][>][l][\lambda + \alpha_i][black][0][p_{i,k}(\lambda)]
\end{align*}
From the definitions in Theorem~$\ref{main}$ we also know that 
\begin{align*}
    \LLn[n]\left(\CAPweb[1][>][0]\right) = \left(-\lambda_i \left(\mu_n + \frac{1}{2}\right) - \frac{n+1}{2}\right) \CAPweb[1][>][$n$] + \mu_n \;\; p_{i,n}(\lambda) \CAPweb[1][>][0] - \frac{1}{2} \xsum{k + l = n,} {k \neq 0} \;\; p_{i,k}(\lambda) \CAPweb[1][>][$l$]
\end{align*}
The adjoint relation~\eqref{item_cycbiadjoint-cyc} of $\Ucat$ follows by noting that $(\lambda + \alpha_i)_i = \lambda_i + 2$ and therefore
\begin{align*}
\LLn[n] \left(\hackcenter{\begin{tikzpicture}[scale=0.7]
    \draw[thick, <-](0,0) .. controls ++(0,.6) and ++(0,.6) .. (.75,0) to (.75,-1);
    \draw[thick, ->](0,0) .. controls ++(0,-.6) and ++(0,-.6) .. (-.75,0) to (-.75,1);
    \node at (-.5,-.9) {\scriptsize $\lambda+\alpha_i$};
    \node at (.8,.9) {\scriptsize $\lambda$};
    \node at (.95,-.8) {\tiny $i$};
\end{tikzpicture}}\right)
\;\; = \;\;
    \LLn[n] \left(\hackcenter{\begin{tikzpicture}[scale=0.7]
    \draw[thick, ->](0,-1)   to (0,1);
    \node at (-.8,.4) {\scriptsize $\lambda+\alpha_i$};
    \node at (.5,.4) {\scriptsize $\lambda$};
    \node at (-.2,-.8) {\tiny $i$};
\end{tikzpicture}}\right)
\;\; = \;\;
0
\end{align*} 

Quadratic KLR relations~\eqref{eq_r2_ij-gen-cyc} are important as they help us fix the parameters of the categorified quantum group (see Remark~\ref{weylactionstandard}). Let us look at the $i \cdot j = -1$ case closely. The $i = j$ case is analogous, and the last case is obvious. We will use the KLR quadratic relation and dot sliding relation \eqref{eq:dotslide} in the following computation.  
\begin{align*}
\LLn[n]\left(\hackcenter{
\begin{tikzpicture}[scale=0.8]
    \draw[thick, ->] (0,0) .. controls ++(0,.5) and ++(0,-.4) .. (.75,.8) .. controls ++(0,.4) and ++(0,-.5) .. (0,1.6);
    \draw[thick, ->] (.75,0) .. controls ++(0,.5) and ++(0,-.4) .. (0,.8) .. controls ++(0,.4) and ++(0,-.5) .. (.75,1.6);
    \node at (1.1,1.25) {\scriptsize $\lambda$};
    \node at (-.2,.1) {\tiny $i$};
    \node at (1.1,.1) {\tiny $j$};
\end{tikzpicture}}\right) 
\;\; &= \;\;
2\mu_n \left( \hackcenter{
\begin{tikzpicture}[scale=0.8]
    \draw[thick, ->] (0,0) .. controls ++(0,.5) and ++(0,-.4) .. (.75,.8) .. controls ++(0,.4) and ++(0,-.5) .. (0,1.6);
    \draw[thick, ->] (.75,0) .. controls ++(0,.5) and ++(0,-.4) .. (0,.8) .. controls ++(0,.4) and ++(0,-.5) .. (.75,1.6);
    \node at (1.1,1.25) {\scriptsize $\lambda$};
    \node at (0,.8) {$\bullet$};
    \node at (0,.8)[left] {\scriptsize $n$};
    \node at (-.2,.1) {\tiny $i$};
    \node at (1.1,.1) {\tiny $j$};
\end{tikzpicture}} - \hackcenter{
\begin{tikzpicture}[scale=0.8]
    \draw[thick, ->] (0,0) .. controls ++(0,.5) and ++(0,-.4) .. (.75,.8) .. controls ++(0,.4) and ++(0,-.5) .. (0,1.6);
    \draw[thick, ->] (.75,0) .. controls ++(0,.5) and ++(0,-.4) .. (0,.8) .. controls ++(0,.4) and ++(0,-.5) .. (.75,1.6);
    \node at (.75,.8) {$\bullet$};
    \node at (.75,.8)[right] {\scriptsize $n$};
    \node at (1.1,1.25) {\scriptsize $\lambda$};
    \node at (-.2,.1) {\tiny $i$};
    \node at (1.1,.1) {\tiny $j$};
\end{tikzpicture}} + \hackcenter{
\begin{tikzpicture}[scale=0.8]
    \draw[thick, ->] (0,0) .. controls ++(0,.5) and ++(0,-.4) .. (.75,.8) .. controls ++(0,.4) and ++(0,-.5) .. (0,1.6);
    \draw[thick, ->] (.75,0) .. controls ++(0,.5) and ++(0,-.4) .. (0,.8) .. controls ++(0,.4) and ++(0,-.5) .. (.75,1.6);
    \node at (0.15,1.254) {$\bullet$};
    \node at (0.15,1.254)[left] {\scriptsize $n$};
    \node at (1.1,1.25) {\scriptsize $\lambda$};
    \node at (-.2,.1) {\tiny $i$};
    \node at (1.1,.1) {\tiny $j$};
\end{tikzpicture}} - \hackcenter{
\begin{tikzpicture}[scale=0.8]
    \draw[thick, ->] (0,0) .. controls ++(0,.5) and ++(0,-.4) .. (.75,.8) .. controls ++(0,.4) and ++(0,-.5) .. (0,1.6);
    \draw[thick, ->] (.75,0) .. controls ++(0,.5) and ++(0,-.4) .. (0,.8) .. controls ++(0,.4) and ++(0,-.5) .. (.75,1.6);
    \node at (0.6,1.254) {$\bullet$};
    \node at (0.6,1.254)[right] {\scriptsize $n$};
    \node at  (1.1,0.8)  {\scriptsize $\lambda$};
    \node at (-.2,.1) {\tiny $i$};
    \node at (1.1,.1) {\tiny $j$};
\end{tikzpicture}}    \right) - \\
&\quad - \frac{1}{2} \sum_{k+l = n} \left( \hackcenter{
\begin{tikzpicture}[scale=0.8]
    \draw[thick, ->] (0,0) .. controls ++(0,.5) and ++(0,-.4) .. (.75,.8) .. controls ++(0,.4) and ++(0,-.5) .. (0,1.6);
    \draw[thick, ->] (.75,0) .. controls ++(0,.5) and ++(0,-.4) .. (0,.8) .. controls ++(0,.4) and ++(0,-.5) .. (.75,1.6);
        \node at (0.6,1.254) {$\bullet$};
    \node at (0.6,1.254)[right] {\scriptsize $l$};
    \node at (0.15,1.254) {$\bullet$};
    \node at (0.15,1.254)[left] {\scriptsize $k$};
    \node at (1.1,0.8) {\scriptsize $\lambda$};
    \node at (-.2,.1) {\tiny $i$};
    \node at (1.1,.1) {\tiny $j$};
\end{tikzpicture}}   \right) - \frac{1}{2} \sum_{k+l = n} \left(  \hackcenter{
\begin{tikzpicture}[scale=0.8]
    \draw[thick, ->] (0,0) .. controls ++(0,.5) and ++(0,-.4) .. (.75,.8) .. controls ++(0,.4) and ++(0,-.5) .. (0,1.6);
    \draw[thick, ->] (.75,0) .. controls ++(0,.5) and ++(0,-.4) .. (0,.8) .. controls ++(0,.4) and ++(0,-.5) .. (.75,1.6);
    \node at (1.1,1.25) {\scriptsize $\lambda$};
    \node at (0,.8) {$\bullet$};
    \node at (.75,.8) {$\bullet$};
    \node at (.75,.8)[right] {\scriptsize $k$};
    \node at (0,.8)[left] {\scriptsize $l$};
    \node at (-.2,.1) {\tiny $i$};
    \node at (1.1,.1) {\tiny $j$};
\end{tikzpicture}} \right) \\
&= - \sum_{k+l = n} \left( \hackcenter{
\begin{tikzpicture}[scale=0.8]
    \draw[thick, ->] (0,0) .. controls ++(0,.5) and ++(0,-.4) .. (.75,.8) .. controls ++(0,.4) and ++(0,-.5) .. (0,1.6);
    \draw[thick, ->] (.75,0) .. controls ++(0,.5) and ++(0,-.4) .. (0,.8) .. controls ++(0,.4) and ++(0,-.5) .. (.75,1.6);
        \node at (0.6,1.254) {$\bullet$};
    \node at (0.6,1.254)[right] {\scriptsize $l$};
    \node at (0.15,1.254) {$\bullet$};
    \node at (0.15,1.254)[left] {\scriptsize $k$};
    \node at (1.1,0.8) {\scriptsize $\lambda$};
    \node at (-.2,.1) {\tiny $i$};
    \node at (1.1,.1) {\tiny $j$};
\end{tikzpicture}}   \right)  
 = \sum_{k+ l =n} -t_{i,j}
  \;      \hackcenter{ 
\begin{tikzpicture}[scale=0.8]
    \draw[thick, ->] (0,0) to (0,1.6);
    \draw[thick, ->] (.75,0) to (.75,1.6);
    \node at (1.1,1.25) { $\lambda$};
    \node at (0,.8){$\bullet$};
    \node at (0,.8)[left] {\scriptsize $l + 1$};
    \node at (.75,.8)[right] {\scriptsize $k$};
    \node at (.75,.8) {$\bullet$};
    \node at (-.2,.1) {\tiny $i$};
    \node at (.95,.1) {\tiny $j$};
\end{tikzpicture}}
  \;\; - \;\; t_{j, i} \;
 \hackcenter{
\begin{tikzpicture}[scale=0.8]
    \draw[thick, ->] (0,0) to (0,1.6);
    \draw[thick, ->] (.75,0) to (.75,1.6);
    \node at (1.1,1.25) { $\lambda$}; 
    \node at (0,.8){$\bullet$};
    \node at (0,.8)[left] {\scriptsize $l$};
    \node at (.75,.8)[right] {\scriptsize $k+1$};
    \node at (.75,.8) {$\bullet$};
    \node at (-.2,.1) {\tiny $i$};
    \node at (.95,.1) {\tiny $j$};
\end{tikzpicture}} 
\\
&= -t_{i,j}
  \;      \hackcenter{ 
\begin{tikzpicture}[scale=0.8]
    \draw[thick, ->] (0,0) to (0,1.6);
    \draw[thick, ->] (.75,0) to (.75,1.6);
    \node at (1.1,1.25) { $\lambda$};
    \node at (0,.8){$\bullet$};
    \node at (0,.8)[left] {\scriptsize $n + 1$};
    \node at (-.2,.1) {\tiny $i$};
    \node at (.95,.1) {\tiny $j$};
\end{tikzpicture}}
  \;\; - \;\; t_{j, i} \;
 \hackcenter{
\begin{tikzpicture}[scale=0.8]
    \draw[thick, ->] (0,0) to (0,1.6);
    \draw[thick, ->] (.75,0) to (.75,1.6);
    \node at (1.1,1.25) { $\lambda$}; 
    \node at (.75,.8)[right] {\scriptsize $n+1$};
    \node at (.75,.8) {$\bullet$};
    \node at (-.2,.1) {\tiny $i$};
    \node at (.95,.1) {\tiny $j$};
\end{tikzpicture}}   = \LLn[n]\left(  t_{i,j}
  \;      \hackcenter{ 
\begin{tikzpicture}[scale=0.8]
    \draw[thick, ->] (0,0) to (0,1.6);
    \draw[thick, ->] (.75,0) to (.75,1.6);
    \node at (1.1,1.25) { $\lambda$};
    \node at (0,.8){$\bullet$};
    \node at (-.2,.1) {\tiny $i$};
    \node at (.95,.1) {\tiny $j$};
\end{tikzpicture}}
  \;\; + \;\; t_{j, i} \;
 \hackcenter{
\begin{tikzpicture}[scale=0.8]
    \draw[thick, ->] (0,0) to (0,1.6);
    \draw[thick, ->] (.75,0) to (.75,1.6);
    \node at (1.1,1.25) { $\lambda$}; 
    \node at (.75,.8) {$\bullet$};
    \node at (-.2,.1) {\tiny $i$};
    \node at (.95,.1) {\tiny $j$};
\end{tikzpicture}}  \right) 
\end{align*}
Here we used the assumption that $t_{ij} = - t_{ji}$ for $i \cdot j = -1$.   

The most involved relation is the $\mathfrak{sl}_2$ relation~\eqref{eq:sl2}. In~\cite[Equations A.7a, A.7b]{Elias_2016}, there was a so called ``magic cancellation''. Something similar happens in our case. We will look at $i = j$ case closely.    We start with the computation of the sideways crossing term
\begin{equation} \label{eq: Ln-EF}
 \LLn[n] \left( \;\; 
\hackcenter{\begin{tikzpicture}[scale=0.8]
    \draw[thick,<-] (0,0) .. controls ++(0,.5) and ++(0,-.5) .. (.75,1);
    \draw[thick] (.75,0) .. controls ++(0,.5) and ++(0,-.5) .. (0,1);
    \draw[thick, ->] (0,1 ) .. controls ++(0,.5) and ++(0,-.5) .. (.75,2);
    \draw[thick] (.75,1) .. controls ++(0,.5) and ++(0,-.5) .. (0,2);
        \node at (-.2,.15) {\tiny $i$};
    \node at (.95,.15) {\tiny $i$};
     \node at (1.1,1.44) { $\lambda $};
\end{tikzpicture}}
\;\; \right) := 
    \LLn[n] \left( \;\;
\hackcenter{\begin{tikzpicture}[xscale=-1.0, scale=0.7]
    \draw[thick, ->] (0,0) .. controls (0,.5) and (.75,.5) .. (.75,1.0);
    \draw[thick, ->] (.75,-.5) to (.75,0) .. controls (.75,.5) and (0,.5) .. (0,1.0) to (0,1.5);
    \draw[thick] (0,0) .. controls ++(0,-.4) and ++(0,-.4) .. (-.75,0) to (-.75,1.5);
    \draw[thick, ->] (.75,1.0) .. controls ++(0,.4) and ++(0,.4) .. (1.5,1.0) to (1.5,-.5);
    \node at (-1.3,.55) {\scriptsize  $\lambda$};
    \node at (1.75,-.2) {\tiny $i$};
    \node at (.55,-.2) {\tiny $i$};
    \node at (-.9,1.2) {};
    \node at (.25,1.2) {};
\node at (.525,2.5) {
    \begin{tikzpicture}[xscale=1, scale=0.7]
    \draw[thick, ->] (0,0) .. controls (0,.5) and (.75,.5) .. (.75,1.0);
    \draw[thick, ->] (.75,-.5) to (.75,0) .. controls (.75,.5) and (0,.5) .. (0,1.0) to (0,1.5);
    \draw[thick,<-] (0,0) .. controls ++(0,-.4) and ++(0,-.4) .. (-.75,0) to (-.75,1.5);
    \draw[thick] (.75,1.0) .. controls ++(0,.4) and ++(0,.4) .. (1.5,1.0) to (1.5,-.5);
    \node at (-.9,1.2) {};
    \node at (.25,1.2) {};
    \end{tikzpicture}
    };
\end{tikzpicture}}
\;\; \right)
\end{equation}
We will freely omit (power sum) bubble slides~\eqref{eq_powerslide2} and nilHecke dot slides~\eqref{eq:dotslide} and compute that~\eqref{eq: Ln-EF} expands into many terms that can be simplified to the following
\begin{align*}
    &=-2\hackcenter{\begin{tikzpicture}[xscale=-1.0, scale=0.7]
    \draw[thick, ->] (0,0) .. controls (0,.5) and (.75,.5) .. (.75,1.0);
    \draw[thick, ->] (.75,-.5) to (.75,0) .. controls (.75,.5) and (0,.5) .. (0,1.0) to (0,1.5);
    \draw[thick] (0,0) .. controls ++(0,-.4) and ++(0,-.4) .. (-.75,0) to (-.75,1.5);
    \draw[thick, ->] (.75,1.0) .. controls ++(0,.4) and ++(0,.4) .. (1.5,1.0) to (1.5,-.5);
    \node at (-1.3,.55) {\scriptsize  $\lambda$};
    \node at (1.75,-.2) {\tiny $i$};
    \node at (.55,-.2) {\tiny $i$};
    \node at (-.9,1.2) {};
    \node at (.25,1.2) {};
\node at (0.235,2.5) {
    \begin{tikzpicture}[xscale=1, scale=0.7]
    \draw[thick, ->] (0,0) .. controls (0,.5) and (.75,.5) .. (.75,1.0);
    \draw[thick, ->] (.75,-.5) to (.75,0) .. controls (.75,.5) and (0,.5) .. (0,1.0) to (0,1.5);
    \draw[thick,<-] (0,0) .. controls ++(0,-.4) and ++(0,-.4) .. (-.75,0) to (-.75,1.5);
    \draw[thick] (.75,1.0) .. controls ++(0,.4) and ++(0,.4) .. (1.5,1.0) to (1.5,-.5);
    \node at (1.5,.5){$\bullet$};
    \node at (1.5,.5)[right] {\scriptsize $n$};
    \node at (-.9,1.2) {};
    \node at (.25,1.2) {};
      \end{tikzpicture}
    };
\end{tikzpicture}} + \sum_{k+l = n, k \neq 0}  p_{i,k}(\lambda) \hackcenter{\begin{tikzpicture}[xscale=-1.0, scale=0.7]
    \draw[thick, ->] (0,0) .. controls (0,.5) and (.75,.5) .. (.75,1.0);
    \draw[thick, ->] (.75,-.5) to (.75,0) .. controls (.75,.5) and (0,.5) .. (0,1.0) to (0,1.5);
    \draw[thick] (0,0) .. controls ++(0,-.4) and ++(0,-.4) .. (-.75,0) to (-.75,1.5);
    \draw[thick, ->] (.75,1.0) .. controls ++(0,.4) and ++(0,.4) .. (1.5,1.0) to (1.5,-.5);
    \node at (-1.3,.55) {\scriptsize  $\lambda$};
    \node at (1.75,-.2) {\tiny $i$};
    \node at (.55,-.2) {\tiny $i$};
    \node at (-.9,1.2) {};
    \node at (.25,1.2) {};
\node at (0.3,2.5) {
    \begin{tikzpicture}[xscale=1, scale=0.7]
    \draw[thick, ->] (0,0) .. controls (0,.5) and (.75,.5) .. (.75,1.0);
    \draw[thick, ->] (.75,-.5) to (.75,0) .. controls (.75,.5) and (0,.5) .. (0,1.0) to (0,1.5);
    \draw[thick,<-] (0,0) .. controls ++(0,-.4) and ++(0,-.4) .. (-.75,0) to (-.75,1.5);
    \draw[thick] (.75,1.0) .. controls ++(0,.4) and ++(0,.4) .. (1.5,1.0) to (1.5,-.5);
    \node at (1.5,.5){$\bullet$};
    \node at (1.5,.5)[right] {\scriptsize $l$};
    \node at (-.9,1.2) {};
    \node at (.25,1.2) {};
      \end{tikzpicture}
    };
\end{tikzpicture}}
-2 \sum_{k+l = n, k \neq 0} \hackcenter{\begin{tikzpicture}[xscale=-1.0, scale=0.7]
    \draw[thick, ->] (0,0) .. controls (0,.5) and (.75,.5) .. (.75,1.0);
    \draw[thick, ->] (.75,-.5) to (.75,0) .. controls (.75,.5) and (0,.5) .. (0,1.0) to (0,1.5);
    \draw[thick] (0,0) .. controls ++(0,-.4) and ++(0,-.4) .. (-.75,0) to (-.75,1.5);
    \draw[thick, ->] (.75,1.0) .. controls ++(0,.4) and ++(0,.4) .. (1.5,1.0) to (1.5,-.5);
    \node at (-1.3,.55) {\scriptsize  $\lambda$};
    \node at (0, 1){$\bullet$};
    \node at (0, 1)[right]{\scriptsize $k$};
    \node at (1.75,-.2) {\tiny $i$};
    \node at (.55,-.2) {\tiny $i$};
    \node at (-.9,1.2) {};
    \node at (.25,1.2) {};
\node at (0.3,2.5) {
    \begin{tikzpicture}[xscale=1, scale=0.7]
    \draw[thick, ->] (0,0) .. controls (0,.5) and (.75,.5) .. (.75,1.0);
    \draw[thick, ->] (.75,-.5) to (.75,0) .. controls (.75,.5) and (0,.5) .. (0,1.0) to (0,1.5);
    \draw[thick,<-] (0,0) .. controls ++(0,-.4) and ++(0,-.4) .. (-.75,0) to (-.75,1.5);
    \draw[thick] (.75,1.0) .. controls ++(0,.4) and ++(0,.4) .. (1.5,1.0) to (1.5,-.5);
    \node at (1.5,.5){$\bullet$};
    \node at (1.5,.5)[right] {\scriptsize $l$};
    \node at (-.9,1.2) {};
    \node at (.25,1.2) {};
      \end{tikzpicture}
    };
\end{tikzpicture}} + \\
&+2 \sum_{k+l = n} \hackcenter{\begin{tikzpicture}[xscale=-1.0, scale=0.7]
    \draw[thick, ->] (0,0) .. controls (0,.5) and (.75,.5) .. (.75,1.0);
    \draw[thick, ->] (.75,-.5) to (.75,0) .. controls (.75,.5) and (0,.5) .. (0,1.0) to (0,1.5);
    \draw[thick] (0,0) .. controls ++(0,-.4) and ++(0,-.4) .. (-.75,0) to (-.75,1.5);
    \draw[thick, ->] (.75,1.0) .. controls ++(0,.4) and ++(0,.4) .. (1.5,1.0) to (1.5,-.5);
    \node at (-1.3,.55) {\scriptsize  $\lambda$};
    \node at (0, 1){$\bullet$};
    \node at (0, 1)[right]{\scriptsize $k$};
    \node at (1.75,-.2) {\tiny $i$};
    \node at (.55,-.2) {\tiny $i$};
    \node at (-.9,1.2) {};
    \node at (.25,1.2) {};
\node at (0.3,2.5) {
    \begin{tikzpicture}[xscale=1, scale=0.7]
    \draw[thick, ->] (0,0) .. controls (0,.5) and (.75,.5) .. (.75,1.0);
    \draw[thick, ->] (.75,-.5) to (.75,0) .. controls (.75,.5) and (0,.5) .. (0,1.0) to (0,1.5);
    \draw[thick,<-] (0,0) .. controls ++(0,-.4) and ++(0,-.4) .. (-.75,0) to (-.75,1.5);
    \draw[thick] (.75,1.0) .. controls ++(0,.4) and ++(0,.4) .. (1.5,1.0) to (1.5,-.5);
    \node at (1.5,.5){$\bullet$};
    \node at (1.5,.5)[right] {\scriptsize $l$};
    \node at (-.9,1.2) {};
    \node at (.25,1.2) {};
      \end{tikzpicture}
    };
\end{tikzpicture}} + \sum_{k+ l = n, k \neq n} \sum_{a + b = l-1} 
\hackcenter{\begin{tikzpicture}[xscale=-1.0, scale=0.7]
    \draw[thick] (.75,-0.5) to (.75,1.0);
    \draw[thick, ->]  (0,0) to (0,1.5);
    \draw[thick] (0,0) .. controls ++(0,-.4) and ++(0,-.4) .. (-.75,0) to (-.75,1.5);
    \draw[thick, ->] (.75,1.0) .. controls ++(0,.4) and ++(0,.4) .. (1.5,1.0) to (1.5,-.5);
    \node at (-1.3,.55) {\scriptsize  $\lambda$};
    \node at (1.75,-.2) {\tiny $i$};
    \node at (.25,-.2) {\tiny $i$};
    \node at (0.75, 1){$\bullet$};
    \node at (0.75, 1)[right]{\scriptsize $b$};
    \node at (-.9,1.2) {};
    \node at (.25,1.2) {};
\node at (-0.105,2.5) {
    \begin{tikzpicture}[xscale=1, scale=0.7]
    \draw[thick, ->] (0,0) .. controls (0,.5) and (.75,.5) .. (.75,1.0);
    \draw[thick, ->] (.75,-.5) to (.75,0) .. controls (.75,.5) and (0,.5) .. (0,1.0) to (0,1.5);
    \draw[thick,<-] (0,0) .. controls ++(0,-.4) and ++(0,-.4) .. (-.75,0) to (-.75,1.5);
    \draw[thick] (.75,1.0) .. controls ++(0,.4) and ++(0,.4) .. (1.5,1.0) to (1.5,-.5);
    \node at (1.5,.5){$\bullet$};
    \node at (1.5,.5)[right] {\scriptsize $k+a$};
    \node at (-.9,1.2) {};
    \node at (.25,1.2) {};
      \end{tikzpicture}
    };
\end{tikzpicture}} + \hackcenter{\begin{tikzpicture}[xscale=-1.0, scale=0.7]
    \draw[thick, ->] (0,0) .. controls (0,.5) and (.75,.5) .. (.75,1.0);
    \draw[thick, ->] (.75,-.5) to (.75,0) .. controls (.75,.5) and (0,.5) .. (0,1.0) to (0,1.5);
    \draw[thick] (0,0) .. controls ++(0,-.4) and ++(0,-.4) .. (-.75,0) to (-.75,1.5);
    \draw[thick, ->] (.75,1.0) .. controls ++(0,.4) and ++(0,.4) .. (1.5,1.0) to (1.5,-.5);
    \node at (-1.3,.55) {\scriptsize  $\lambda$};
    \node at (1.75,-.2) {\tiny $i$};
    \node at (.55,-.2) {\tiny $i$};
    \node at (-.9,1.2) {};
    \node at (.25,1.2) {};
\node at (-0.105,2.5) {
    \begin{tikzpicture}[xscale=1, scale=0.7]
    \draw[thick] (0,0) to (0,1.5);
    \draw[thick] (.75,-.5) to (.75,1);
    \draw[thick,<-] (0,0) .. controls ++(0,-.4) and ++(0,-.4) .. (-.75,0) to (-.75,1.5);
    \draw[thick] (.75,1.0) .. controls ++(0,.4) and ++(0,.4) .. (1.5,1.0) to (1.5,-.5);
    \node at (0, 1){$\bullet$};
    \node at (0, 1)[left]{\scriptsize $b$};
    \node at (1.5,.5){$\bullet$};
    \node at (1.5,.5)[right] {\scriptsize $k+a$};
    \node at (-.9,1.2) {};
    \node at (.25,1.2) {};
    \end{tikzpicture}
    };
\end{tikzpicture}}  
\\
&+\left(-\lambda_i\left(\mu_n + \frac{1}{2}\right) - \frac{n+1}{2}\right) \sum_{k +l = n -1}
\hackcenter{\begin{tikzpicture}[xscale=-1.0, scale=0.7]
    \draw[thick] (.75,-0.5) to (.75,1.0);
    \draw[thick, ->]  (0,0) to (0,1.5);
    \draw[thick] (0,0) .. controls ++(0,-.4) and ++(0,-.4) .. (-.75,0) to (-.75,1.5);
    \draw[thick, ->] (.75,1.0) .. controls ++(0,.4) and ++(0,.4) .. (1.5,1.0) to (1.5,-.5);
    \node at (-1.3,.55) {\scriptsize  $\lambda$};
    \node at (1.75,-.2) {\tiny $i$};
    \node at (.25,-.2) {\tiny $i$};
    \node at (0, 1){$\bullet$};
    \node at (0, 1)[right]{\scriptsize $k$};
    \node at (0.75, 1){$\bullet$};
    \node at (0.75, 1)[right]{\scriptsize $l$};
    \node at (-.9,1.2) {};
    \node at (.25,1.2) {};
\node at (0.525,2.5) {
    \begin{tikzpicture}[xscale=1, scale=0.7]
    \draw[thick, ->] (0,0) .. controls (0,.5) and (.75,.5) .. (.75,1.0);
    \draw[thick, ->] (.75,-.5) to (.75,0) .. controls (.75,.5) and (0,.5) .. (0,1.0) to (0,1.5);
    \draw[thick,<-] (0,0) .. controls ++(0,-.4) and ++(0,-.4) .. (-.75,0) to (-.75,1.5);
    \draw[thick] (.75,1.0) .. controls ++(0,.4) and ++(0,.4) .. (1.5,1.0) to (1.5,-.5);
    \node at (-.9,1.2) {};
    \node at (.25,1.2) {};
      \end{tikzpicture}
    };
\end{tikzpicture}}  - \frac{1}{2} \xsum{k+l = n,}{k \neq 0} p_{i,k}(\lambda)
\hackcenter{\begin{tikzpicture}[xscale=-1.0, scale=0.7]
    \draw[thick] (0,0) .. controls (0,.5) and (.75,.5) .. (.75,1.0);
    \draw[thick, ->] (.75,-.5) to (.75,0) .. controls (.75,.5) and (0,.5) .. (0,1.0) to (0,1.5);
    \draw[thick] (0,0) .. controls ++(0,-.4) and ++(0,-.4) .. (-.75,0) to (-.75,1.5);
    \draw[thick, ->] (.75,1.0) .. controls ++(0,.4) and ++(0,.4) .. (1.5,1.0) to (1.5,-.5);
    \node at (0.75, 1){$\bullet$};
    \node at (0.75, 1)[right]{\scriptsize $l$};
    \node at (-1.3,.55) {\scriptsize  $\lambda$};
    \node at (1.75,-.2) {\tiny $i$};
    \node at (.55,-.2) {\tiny $i$};
    \node at (-.9,1.2) {};
    \node at (.25,1.2) {};
\node at (.525,2.5) {
    \begin{tikzpicture}[xscale=1, scale=0.7]
    \draw[thick, ->] (0,0) .. controls (0,.5) and (.75,.5) .. (.75,1.0);
    \draw[thick, ->] (.75,-.5) to (.75,0) .. controls (.75,.5) and (0,.5) .. (0,1.0) to (0,1.5);
    \draw[thick,<-] (0,0) .. controls ++(0,-.4) and ++(0,-.4) .. (-.75,0) to (-.75,1.5);
    \draw[thick] (.75,1.0) .. controls ++(0,.4) and ++(0,.4) .. (1.5,1.0) to (1.5,-.5);
    \node at (-.9,1.2) {};
    \node at (.25,1.2) {};
    \end{tikzpicture}
    };
\end{tikzpicture}} \\ 
&+ \left(\lambda_i\left(\mu_n - \frac{1}{2}\right) - \frac{n+1}{2}\right) \sum_{k +l = n -1}
\hackcenter{\begin{tikzpicture}[xscale=-1.0, scale=0.7]
    \draw[thick, ->] (0,0) .. controls (0,.5) and (.75,.5) .. (.75,1.0);
    \draw[thick, ->] (.75,-.5) to (.75,0) .. controls (.75,.5) and (0,.5) .. (0,1.0) to (0,1.5);
    \draw[thick] (0,0) .. controls ++(0,-.4) and ++(0,-.4) .. (-.75,0) to (-.75,1.5);
    \draw[thick, ->] (.75,1.0) .. controls ++(0,.4) and ++(0,.4) .. (1.5,1.0) to (1.5,-.5);
    \node at (-1.3,.55) {\scriptsize  $\lambda$};
    \node at (1.75,-.2) {\tiny $i$};
    \node at (.55,-.2) {\tiny $i$};
    \node at (-.9,1.2) {};
    \node at (.25,1.2) {};
\node at (.525,2.5) {
    \begin{tikzpicture}[xscale=1, scale=0.7]
    \draw[thick] (0,0) to (0,1.5);
    \draw[thick] (.75,-.5) to (.75,1);
    \draw[thick,<-] (0,0) .. controls ++(0,-.4) and ++(0,-.4) .. (-.75,0) to (-.75,1.5);
    \draw[thick] (.75,1.0) .. controls ++(0,.4) and ++(0,.4) .. (1.5,1.0) to (1.5,-.5);
    \node at (0, 1){$\bullet$};
    \node at (0, 1)[left]{\scriptsize $k$};
    \node at (0.75, 1){$\bullet$};
    \node at (0.75, 1)[left]{\scriptsize $l$};
    \node at (-.9,1.2) {};
    \node at (.25,1.2) {};
    \end{tikzpicture}
    };
\end{tikzpicture}} - \frac{1}{2}\xsum{k+l = n,}{k \neq 0} p_{i,k}(\lambda)
\hackcenter{\begin{tikzpicture}[xscale=-1.0, scale=0.7]
    \draw[thick, ->] (0,0) .. controls (0,.5) and (.75,.5) .. (.75,1.0);
    \draw[thick, ->] (.75,-.5) to (.75,0) .. controls (.75,.5) and (0,.5) .. (0,1.0) to (0,1.5);
    \draw[thick] (0,0) .. controls ++(0,-.4) and ++(0,-.4) .. (-.75,0) to (-.75,1.5);
    \draw[thick, ->] (.75,1.0) .. controls ++(0,.4) and ++(0,.4) .. (1.5,1.0) to (1.5,-.5);
    \node at (-1.3,.55) {\scriptsize  $\lambda$};
    \node at (1.75,-.2) {\tiny $i$};
    \node at (.55,-.2) {\tiny $i$};
    \node at (-.9,1.2) {};
    \node at (.25,1.2) {};
\node at (.525,2.5) {
    \begin{tikzpicture}[xscale=1, scale=0.7]
    \draw[thick, ->] (0,0) .. controls (0,.5) and (.75,.5) .. (.75,1.0);
    \draw[thick, ->] (.75,-.5) to (.75,0) .. controls (.75,.5) and (0,.5) .. (0,1.0) to (0,1.5);
    \draw[thick] (0,0) .. controls ++(0,-.4) and ++(0,-.4) .. (-.75,0) to (-.75,1.5);
    \draw[thick] (.75,1.0) .. controls ++(0,.4) and ++(0,.4) .. (1.5,1.0) to (1.5,-.5);
    \node at (0.0, 0){$\bullet$};
    \node at (0.0, 0)[right]{\scriptsize $l$};
    \node at (-.9,1.2) {};
    \node at (.25,1.2) {};
    \end{tikzpicture}
    };
\end{tikzpicture}}
\end{align*}
This can be further simplified to
\begin{align*}
&= \left(-\lambda_i\left(\mu_n + \frac{1}{2}\right) - \frac{n+1}{2} + (n - l)\right) \sum_{k +l = n -1}
\hackcenter{\begin{tikzpicture}[xscale=-1.0, scale=0.7]
    \draw[thick] (.75,-0.5) to (.75,1.0);
    \draw[thick, ->]  (0,0) to (0,1.5);
    \draw[thick] (0,0) .. controls ++(0,-.4) and ++(0,-.4) .. (-.75,0) to (-.75,1.5);
    \draw[thick, ->] (.75,1.0) .. controls ++(0,.4) and ++(0,.4) .. (1.5,1.0) to (1.5,-.5);
    \node at (-1.3,.55) {\scriptsize  $\lambda$};
    \node at (1.75,-.2) {\tiny $i$};
    \node at (.25,-.2) {\tiny $i$};
    \node at (0, 1){$\bullet$};
    \node at (0, 1)[right]{\scriptsize $k$};
    \node at (0.75, 1){$\bullet$};
    \node at (0.75, 1)[right]{\scriptsize $l$};
    \node at (-.9,1.2) {};
    \node at (.25,1.2) {};
\node at (0.525,2.5) {
    \begin{tikzpicture}[xscale=1, scale=0.7]
    \draw[thick, ->] (0,0) .. controls (0,.5) and (.75,.5) .. (.75,1.0);
    \draw[thick, ->] (.75,-.5) to (.75,0) .. controls (.75,.5) and (0,.5) .. (0,1.0) to (0,1.5);
    \draw[thick,<-] (0,0) .. controls ++(0,-.4) and ++(0,-.4) .. (-.75,0) to (-.75,1.5);
    \draw[thick] (.75,1.0) .. controls ++(0,.4) and ++(0,.4) .. (1.5,1.0) to (1.5,-.5);
    \node at (-.9,1.2) {};
    \node at (.25,1.2) {};
      \end{tikzpicture}
    };
\end{tikzpicture}} 
- \frac{1}{2} \xsum{k+l = n,}{k \neq 0} \sum_{a + b = l - 1}  p_{i,k}(\lambda) \hackcenter{\begin{tikzpicture}[xscale=-1.0, scale=0.7]
    \draw[thick] (.75,-0.5) to (.75,1.0);
    \draw[thick, ->]  (0,0) to (0,1.5);
    \draw[thick] (0,0) .. controls ++(0,-.4) and ++(0,-.4) .. (-.75,0) to (-.75,1.5);
    \draw[thick, ->] (.75,1.0) .. controls ++(0,.4) and ++(0,.4) .. (1.5,1.0) to (1.5,-.5);
    \node at (-1.3,.55) {\scriptsize  $\lambda$};
    \node at (1.75,-.2) {\tiny $i$};
    \node at (.25,-.2) {\tiny $i$};
    \node at (0, 1){$\bullet$};
    \node at (0, 1)[right]{\scriptsize $a$};
    \node at (0.75, 1){$\bullet$};
    \node at (0.75, 1)[right]{\scriptsize $b$};
    \node at (-.9,1.2) {};
    \node at (.25,1.2) {};
\node at (0.525,2.5) {
    \begin{tikzpicture}[xscale=1, scale=0.7]
    \draw[thick, ->] (0,0) .. controls (0,.5) and (.75,.5) .. (.75,1.0);
    \draw[thick, ->] (.75,-.5) to (.75,0) .. controls (.75,.5) and (0,.5) .. (0,1.0) to (0,1.5);
    \draw[thick,<-] (0,0) .. controls ++(0,-.4) and ++(0,-.4) .. (-.75,0) to (-.75,1.5);
    \draw[thick] (.75,1.0) .. controls ++(0,.4) and ++(0,.4) .. (1.5,1.0) to (1.5,-.5);
    \node at (-.9,1.2) {};
    \node at (.25,1.2) {};
      \end{tikzpicture}
    };
\end{tikzpicture}}
\\ 
& \quad + \left(\lambda_i \left(\mu_n - \frac{1}{2}\right) - \frac{n+1}{2} + (n - k)\right) \sum_{k +l = n -1}
\hackcenter{\begin{tikzpicture}[xscale=-1.0, scale=0.7]
    \draw[thick, ->] (0,0) .. controls (0,.5) and (.75,.5) .. (.75,1.0);
    \draw[thick, ->] (.75,-.5) to (.75,0) .. controls (.75,.5) and (0,.5) .. (0,1.0) to (0,1.5);
    \draw[thick] (0,0) .. controls ++(0,-.4) and ++(0,-.4) .. (-.75,0) to (-.75,1.5);
    \draw[thick, ->] (.75,1.0) .. controls ++(0,.4) and ++(0,.4) .. (1.5,1.0) to (1.5,-.5);
    \node at (-1.3,.55) {\scriptsize  $\lambda$};
    \node at (1.75,-.2) {\tiny $i$};
    \node at (.55,-.2) {\tiny $i$};
    \node at (-.9,1.2) {};
    \node at (.25,1.2) {};
\node at (.525,2.5) {
    \begin{tikzpicture}[xscale=1, scale=0.7]
    \draw[thick] (0,0) to (0,1.5);
    \draw[thick] (.75,-.5) to (.75,1);
    \draw[thick,<-] (0,0) .. controls ++(0,-.4) and ++(0,-.4) .. (-.75,0) to (-.75,1.5);
    \draw[thick] (.75,1.0) .. controls ++(0,.4) and ++(0,.4) .. (1.5,1.0) to (1.5,-.5);
    \node at (0, 1){$\bullet$};
    \node at (0, 1)[left]{\scriptsize $k$};
    \node at (0.75, 1){$\bullet$};
    \node at (0.75, 1)[left]{\scriptsize $l$};
    \node at (-.9,1.2) {};
    \node at (.25,1.2) {};
    \end{tikzpicture}
    };
\end{tikzpicture}}
- \frac{1}{2}\xsum{k+l = n,}{k \neq 0} \sum_{a + b = l - 1} p_{i,k}(\lambda) 
\hackcenter{\begin{tikzpicture}[xscale=-1.0, scale=0.7]
    \draw[thick, ->] (0,0) .. controls (0,.5) and (.75,.5) .. (.75,1.0);
    \draw[thick, ->] (.75,-.5) to (.75,0) .. controls (.75,.5) and (0,.5) .. (0,1.0) to (0,1.5);
    \draw[thick] (0,0) .. controls ++(0,-.4) and ++(0,-.4) .. (-.75,0) to (-.75,1.5);
    \draw[thick, ->] (.75,1.0) .. controls ++(0,.4) and ++(0,.4) .. (1.5,1.0) to (1.5,-.5);
    \node at (-1.3,.55) {\scriptsize  $\lambda$};
    \node at (1.75,-.2) {\tiny $i$};
    \node at (.55,-.2) {\tiny $i$};
    \node at (-.9,1.2) {};
    \node at (.25,1.2) {};
\node at (.525,2.5) {
    \begin{tikzpicture}[xscale=1, scale=0.7]
    \draw[thick] (0,0) to (0,1.5);
    \draw[thick] (.75,-.5) to (.75,1);
    \draw[thick,<-] (0,0) .. controls ++(0,-.4) and ++(0,-.4) .. (-.75,0) to (-.75,1.5);
    \draw[thick] (.75,1.0) .. controls ++(0,.4) and ++(0,.4) .. (1.5,1.0) to (1.5,-.5);
    \node at (0, 1){$\bullet$};
    \node at (0, 1)[left]{\scriptsize $b$};
    \node at (0.75, 1){$\bullet$};
    \node at (0.75, 1)[left]{\scriptsize $a$};
    \node at (-.9,1.2) {};
    \node at (.25,1.2) {};
    \end{tikzpicture}
    };
\end{tikzpicture}}
\end{align*}

We may now use~\cite[Equations 5.23, 5.24]{Lau1}  to further simplify the expression. We would like to note that this agrees with the computation of~\cite[Equations A.7a, A.7b]{Elias_2016}, for the case of $n = 1$ (our conventions differ by a sign).
\begin{align*}
    &=\frac{1}{2} \sum_{k+l = n, k \neq 0} \sum_{a + b = l - 1}  \sum_{c + d = a - \lambda_i}  p_{i,k}(\lambda)
    \hackcenter{\begin{tikzpicture}[xscale=-1.0, scale=0.7]
    \draw[thick] (.75,-0.5) to (.75,1.0);
    %
    \draw[thick, ->] (.75,1.0) .. controls ++(0,.4) and ++(0,.4) .. (1.5,1.0) to (1.5,-.5);
    \node at (1.75,-.2) {\tiny $i$};
    \node at (0.75, 1){$\bullet$};
    \node at (0.75, 1)[right]{\scriptsize $b$};
\node at (1,2.5) {
    \begin{tikzpicture}[xscale=1, scale=0.7]
    \draw[thick] (0,0) to (0,1.5);
    \draw[thick,<-] (0,0) .. controls ++(0,-.4) and ++(0,-.4) .. (-.75,0) to (-.75,1.5);
    \node at (0, 1){$\bullet$};
    \node at (0, 1)[right]{\scriptsize $d$};
    \node at (-0.75,1.2)[left] {\tiny $i$};
      \end{tikzpicture}
    };
\end{tikzpicture}}
\CIRCLEweb[0.5][<][c][\lambda][black][0] +  p_{i,k}(\lambda)
    \hackcenter{\begin{tikzpicture}[xscale=-1.0, scale=0.7]
    \draw[thick] (.75,-0.5) to (.75,1.0);
    %
    \draw[thick, ->] (.75,1.0) .. controls ++(0,.4) and ++(0,.4) .. (1.5,1.0) to (1.5,-.5);
    \node at (1.75,-.2) {\tiny $i$};
    \node at (0.75, 1){$\bullet$};
    \node at (0.75, 1)[right]{\scriptsize $d$};
\node at (1,2.5) {
    \begin{tikzpicture}[xscale=1, scale=0.7]
    \draw[thick] (0,0) to (0,1.5);
    \draw[thick,<-] (0,0) .. controls ++(0,-.4) and ++(0,-.4) .. (-.75,0) to (-.75,1.5);
    \node at (0, 1){$\bullet$};
    \node at (0, 1)[right]{\scriptsize $b$};
    \node at (-0.75,1.2)[left] {\tiny $i$};
      \end{tikzpicture}
    };
\end{tikzpicture}}
\CIRCLEweb[0.5][<][c][\lambda][black][0] \\
&-\left(-\lambda_i\left(\mu_n + \frac{1}{2}\right) - \frac{n+1}{2} + (n - l)\right) \sum_{k +l = n -1}   \sum_{a + b = k - \lambda_i} 
    \hackcenter{\begin{tikzpicture}[xscale=-1.0, scale=0.7]
    \draw[thick] (.75,-0.5) to (.75,1.0);
    %
    \draw[thick, ->] (.75,1.0) .. controls ++(0,.4) and ++(0,.4) .. (1.5,1.0) to (1.5,-.5);
    \node at (1.75,-.2) {\tiny $i$};
    \node at (0.75, 1){$\bullet$};
    \node at (0.75, 1)[right]{\scriptsize $l$};
\node at (1,2.5) {
    \begin{tikzpicture}[xscale=1, scale=0.7]
    \draw[thick] (0,0) to (0,1.5);
    \draw[thick,<-] (0,0) .. controls ++(0,-.4) and ++(0,-.4) .. (-.75,0) to (-.75,1.5);
    \node at (0, 1){$\bullet$};
    \node at (0, 1)[right]{\scriptsize $a$};
    \node at (-0.75,1.2)[left] {\tiny $i$};
      \end{tikzpicture}
    };
\end{tikzpicture}}
\CIRCLEweb[0.5][<][b][\lambda][black][0] - \\
 &- \left(\lambda_i\left(\mu_n - \frac{1}{2}\right) - \frac{n+1}{2} + (n - k)\right) \sum_{k +l = n -1} \sum_{a + b = l - \lambda_i} 
    \hackcenter{\begin{tikzpicture}[xscale=-1.0, scale=0.7]
    \draw[thick] (.75,-0.5) to (.75,1.0);
    %
    \draw[thick, ->] (.75,1.0) .. controls ++(0,.4) and ++(0,.4) .. (1.5,1.0) to (1.5,-.5);
    \node at (1.75,-.2) {\tiny $i$};
    \node at (0.75, 1){$\bullet$};
    \node at (0.75, 1)[right]{\scriptsize $a$};
\node at (1,2.5) {
    \begin{tikzpicture}[xscale=1, scale=0.7]
    \draw[thick] (0,0) to (0,1.5);
    \draw[thick,<-] (0,0) .. controls ++(0,-.4) and ++(0,-.4) .. (-.75,0) to (-.75,1.5);
    \node at (0, 1){$\bullet$};
    \node at (0, 1)[right]{\scriptsize $k$};
    \node at (-0.75,1.2)[left] {\tiny $i$};
      \end{tikzpicture}
    };
\end{tikzpicture}}
\CIRCLEweb[0.5][<][b][\lambda][black][0]
\end{align*}

Using Newton’s identities~\eqref{eq:Newton}, which express complete and elementary symmetric functions in terms of power sum functions and lower-order terms, along with the correspondence between bubbles and symmetric functions from Section~\ref{subsec:symmetric functions}, one can verify that the relation holds. The mixed $EF$ relation~\eqref{mixed_rel-cyc} follows as a simpler special case of the same computation. Newton’s identities can also be applied to establish the bubble relation~\eqref{eq:bubblesarezero},~\eqref{eq:degreezero}, using the result of Lemma~\ref{actiononbubbles}.

Applying $\LLn$ to the left side of the cubic relation~\eqref{eq:KLRqubic} can be simplified using dot slides \eqref{eq:dotslide}, cubic~\eqref{eq:KLRqubic} and quadratic relations~\eqref{eq_r2_ij-gen-cyc} to
\begin{align*}
    &\sum_{k+l = n} \left( t_{ij} \tripleuparrow[$k$][0][$l$] - \frac{1}{2}  t_{ij} \tripleuparrow[$l$][$k$] - \frac{1}{2} t_{ij} \tripleuparrow[$k$][$l$] \right. \\
      & \quad -\frac{1}{2}\sum_{a+b = l-1}  \left( t_{ji} \tripleuparrow[$a$][$k+1$][$b$] + t_{ij} \tripleuparrow[$a$][$k$][$b+1$] \right) \\ 
     &\quad - \left.\frac{1}{2}\sum_{a+b = k-1}  \left( t_{ji} \tripleuparrow[$a$][$l+1$][$b$] + t_{ij} \tripleuparrow[$a$][$l$][$b+1$] \right) \right) \\
\end{align*}
It is not hard to see that for $j\cdot i = -1$, $t_{ij} = -t_{ji}$, all the terms cancel out, agreeing with $\LLn$ of the right-hand-side of~\eqref{eq:KLRqubic}.  
The remaining relations are proven using similar techniques.
\end{proof}

Theorem~\ref{main} follows as the $\LLn$ are well-defined 2-functors by Lemma~\ref{lemmabasicrelation} that satisfy the defining relations of $\ourwitt$ by Lemma~\ref{lem:defining}.

\subsection{Specialization to $\mathfrak{sl}_2$ and $\Ucat(\mathfrak{sl}_2)$}
As a special case, the following proposition gives an $\ourwitt$  action on $\Ucat(\mathfrak{sl}_2)$. By restricting the action from Theorem~\ref{main} to just $\LLn[-1], \LLn[0], \LLn[1]$ we answer a question of~\cite{qi2024symmetriesmathfrakglnfoams}, see the second to last paragraph of Section 4.5.

\begin{proposition}
There is an action of the positive half of the Witt algebra $\ourwitt$ on $\Ucat(\mathfrak{sl}_2)$.  By restricting to $\mathfrak{sl}_2$ inside the positive half of the Witt algebra as in Lemma~\ref{lem:wittsl2}, we get an action of $\mathfrak{sl}_2$ on the 2-morphisms of $\Ucat(\mathfrak{sl}_2)$ that agrees with the action in~\cite[Definition 6.13]{Elias_2023}, see also~\cite[Definition 4.6]{Elias_2016} for more choices of $\LLn[1]$.
\end{proposition}

\begin{proof}
    By restricting the Cartan datum of $\Ucat(\mathfrak{g})$ (via $I = \{1\}, X = \Z$), we obtain a version of $\Ucat(\mathfrak{sl}_2)$, see~\cite{Lau-param} for details, hence Theorem~\ref{main} gives an action of the positive half of the Witt algebra on $\Ucat(\mathfrak{sl}_2)$.

    We recover some of the (degree $2$) derivations in~\cite[Definition 4.6]{Elias_2016}, by setting $n =1, i = 1$, $\lambda_1 = \lambda$, and $(\lambda (\mu_1 + \frac{1}{2}) - 1) = -\overline{x}_{\lambda - 2},  - \mu_1 + \frac{1}{2} = \overline{y}_{\lambda - 2}, -2\mu_1 = a_{\lambda}$. The minus signs arise as the action of $z$ in definition 6.13 of~\cite{Elias_2023} on $\Ucat(\mathfrak{sl}_2)$ carries a different sign compared to $\LLn[-1]$.
\end{proof}

\section{Relation to Witt algebra action on foams}
In this section, we examine how the action constructed in this paper arises from the foam-based action studied in~\cite{qi2024symmetriesmathfrakglnfoams}, via categorical skew Howe duality as developed in~\cite{QR}, and generalizing earlier observations in the $\mathfrak{sl}_2$ and $\mathfrak{sl}_3$ cases~\cite{Lau-skew}. 
To be precise, in~\cite{qi2024symmetriesmathfrakglnfoams} they work with equivariant $\mathfrak{gl}_n$ foams of~\cite{Robert2020-xp}. These correspond to the foams in~\cite{QR} where we do not pass to the specialization named $n\mathsf{Foam}^{\bullet}$ in~\cite[Section 4.1]{QR}. In the world of categorified quantum groups, these foams are related, via categorical skew Howe duality, to deformed cyclotomic quotients $\check{\mathcal{U}}_Q^{\lambda,\ast}$ of~\cite{RouQH,Webster_2017}.  These equivariant foams are also closely related to the equivariant flag categories of~\cite{Lau2,KL3}. 

To match with~\cite{QR}, we will work with $\Ucat := (\check{\mathcal{U}}_Q(\mathfrak{sl}_{\infty})^+)^{n,\ast}$ using the choice of scalars from Remark~\ref{weylactionstandard} and $\mathfrak{sl}_{\infty}$ root datum. Here $\Ucat(\mathfrak{sl}_{\infty})$ is the direct limit of $\Ucat(\mathfrak{sl}_{n})$'s, the check mark corresponds to a certain subcategory of the Karoubi envelope and the plus sign stands for the fact that we restrict the categorified quantum group to objects that are increasing in the sense that for $(\lambda_1, \ldots, \lambda_k) = \lambda \in X$ of Definition~\ref{objects2cat}, it holds that $\lambda_1 \leq \lambda_2 \leq \ldots \leq \lambda_k$. The other subscripts indicate the restriction to the deformed cyclotomic quotient of weight $\lambda$.

We use the notation $n\mathsf{Foam} = \bigoplus_{N \geq 1} n \mathsf{Foam}(N)$ for the category of $\mathfrak{gl}_n$ ladderized foams following~\cite{QR}. Since we restrict our focus to ladderized foams, we do not recover the action on the cap and cup foam from~\cite[Equations 68, 69]{qi2024symmetriesmathfrakglnfoams}.

Our definition of the Witt action of $\ourwitt$ is designed to be compatible with the action on foams, and this compatibility guides many of the choices made in the construction. We illustrate in the proof of Theorem~\ref{equivarancefoam}  by working through several examples where the action on the categorified quantum group recovers the expected behavior under foamation.

A feature of categorical skew Howe duality is that multiple distinct foams may be mapped to the same diagram in the categorified quantum group, depending on how we represent an $\mathfrak{sl}_m$ weight in $\mathfrak{gl}_m$ and on the choice of ladderization, see \cite[Remark 2.3.3]{QR}. We exploit this phenomenon to reduce the number of parameters, thereby ensuring compatibility with the defining relations of the categorified quantum group. Skew Howe duality also suggests that power sum bubbles should be represented using thick bubbles, as in~\cite[Equation 3.56]{QR}. This is equivalent to our approach via~\cite[Equations 4.33 and 4.34]{Khovanov_2012} or~\cite[Equations 5.1 and 5.2]{MSV}, provided the bubbles are sufficiently thick (at least $n$ in the image of the action of $\LLn[n]$), so that the relevant determinants behave well. To avoid these technical constraints, we have opted for the definition given in~\eqref{eq_defpil}. 

\begin{remark}
    Interpreting the power sum bubbles as thick bubbles allows one to reformulate the combinatorics of symmetric functions in terms of hook-shaped Young diagrams, offering an alternative perspective on several of our proofs.
\end{remark}

\begin{thm}\label{equivarancefoam}
    Foamation 2-functors $\Phi$ of~\cite[Section 3.2]{QR} from $\Ucat$  into the category of ladderized foams $n\mathsf{Foam}$, intertwine the action of the positive half of the Witt algebra, that is, make the following diagram commute: 
\[
\begin{tikzcd}
\Ucat \arrow[r, "\LLn"] \arrow[d, "\Phi" description] & \Ucat \arrow[d, "\Phi" description] \\
n\mathsf{Foam} \arrow[r, "\LLn"]                                  & n\mathsf{Foam}                                
\end{tikzcd}
\]
Here $\ourwitt$ action on the top arrow for the categorified quantum group is the one from Theorem~\ref{main} and the action on foams is presented in~\cite{qi2024symmetriesmathfrakglnfoams}, fixing the parameters appearing in their work to $ \lambda_n + \mu_n = 0, s = \frac{1}{2} $.   
\end{thm}

\begin{proof}
Our choice of scalars from Remark~\ref{weylactionstandard} is compatible with the choices of~\cite{QR}.  
Under certain conditions, Skew Howe duality becomes an eventually faithful functor, see~\cite[Proposition 3.22]{QR}, provided we work with $\mathfrak{gl}$ in place of $\mathfrak{sl}$ in $\Ucat$. The action defined in Theorem~\ref{main}, will descend to an action on the foams in the image of $\Phi$, if we show that changing $\mathfrak{sl}$ weights to $\mathfrak{gl}$ weights does not change the action

Furthermore, we will show that the action of $\ourwitt$ is intertwined under $\Phi$. For that, we need to show that our coefficients match up with \cite{qi2024symmetriesmathfrakglnfoams}. We will work with $\Ucat(\mathfrak{sl}_2)$ weights and diagrams, the pictures can be directly realized in $\Ucat$, and the techniques extended to other morphisms in that category. We illustrate the proof by considering the action of $\LLn[1]$ and also clarify its relation to the construction in~\cite{Elias_2016}, from the formulas given in~\cite[Section 4.1]{qi2024symmetriesmathfrakglnfoams}. Note that the foams in\cite{qi2024symmetriesmathfrakglnfoams} need not be ladderized, as mentioned above, we do not recover the action on cap and cup foams. The actions of the other operators $\LLn[i]$ can be recovered in a similar fashion. We adopt the notation of~\cite[Section 4.1]{qi2024symmetriesmathfrakglnfoams}, where $s \in \Bbbk$ and $(\mu_i)_{i \geq -1}, (\lambda_i)_{i \geq -1}$ are two Witt sequences.

In the notation of ~\cite[Section 3.2]{QR}, facets colored blue are taken to have thickness $1$. 
First we consider a ladderization, see~\cite{Cautis_2014} for a definition, for a positive value of $\lambda$ as follows
\begin{align*}
     \LLn[1]  \circ \Phi &{\left(\CUPweb[1][<][0][\lambda][black][0][0] \right)} = \LLn[1] \left( \cupEFfoama[0.5] \right) 
     = (\lambda_1 + s) \cupEFfoama[0.5][$p_1$][0][$p_0$] + (\mu_1 + s)\cupEFfoama[0.5][$p_0$][0][$p_1$] \\
    &= (\lambda_1 + s)(\lambda - 1) \cupEFfoama[0.5][][0][0] + (\mu_1 + s) \left( \cupEFfoama[0.5][0][$p_1$][0]  - \cupEFfoama[0.5][][0][0] \right) 
    \end{align*}
It is straightforward to check that this agrees with the action on the categorified quantum group
\begin{align*}
    \Phi^{-1} \circ \LLn[1] \circ \Phi & {\left( \CUPweb[1][<][0][\lambda][black][0][0]\right)} = ((\lambda_1 + s)(\lambda - 1) - \mu_1 - s) \CUPweb[1][<][1][\lambda][black][0][0] + (\mu_1 + s) \CIRCLEweb[0.5][<][1][0][black][0][0] \CUPweb[1][<][0][\lambda][black][0][0]
\end{align*}

Next, we consider a ladderization for a negative or zero value of $\lambda$ as follows
\begin{align*}
   & \LLn[1]  \circ \Phi {\left( \CUPweb[1][<][0][\lambda][black][0][0] \right)} = 
   \LLn[1] \left( \cupEFfoamb[0.5] \right) 
   \\
   & \quad  = (\lambda_1 + s - 1) \cupEFfoamb[0.5][$p_1$][$p_0$][0] + (\mu_1 + s - 1) \cupEFfoamb[0.5][$p_0$][$p_1$][0] 
   \\
    & \quad = (\lambda_1 + s - 1) \left( \cupEFfoamb[0.5][0][0][$p_1$]  -  \cupEFfoamb[0.5][0][$p_1$][0] \right) + (\mu_1 + s - 1)(-\lambda + 1) \cupEFfoamb[0.5][0][$p_1$][0]
\end{align*}
Now considering the preimage of these foams under foamation $\Phi$
\begin{align*}
    \Phi^{-1} \circ \LLn[1] \circ \Phi  {\left( \CUPweb[1][<][0][\lambda][black][0][0] \right)} &= ((\mu_1 + s - 1)(-\lambda + 1) - \lambda_1 - s + 1) \CUPweb[1][<][1][\lambda][black][0][0]  + (\lambda_1 + s - 1) \CIRCLECUPweb[0.5][>][1][\lambda][0][0][>] \\
     &= ((\mu_1 + s - 1)(-\lambda + 1) +(\lambda_1 + s - 1)) \CUPweb[1][<][1][\lambda][black][0][0] + (\lambda_1 + s - 1)\CIRCLEweb[0.5][>][1][0][black][0][0]\CUPweb[1][<][0][\lambda][black][0][0]\\
    &= ((\mu_1 + s - 1)(-\lambda + 1) +(\lambda_1 + s - 1)) \CUPweb[1][<][1][\lambda][black][0][0] - (\lambda_1 + s - 1) \CIRCLEweb[0.5][<][1][0][black][0][0]\CUPweb[1][<][0][\lambda][black][0][0]
\end{align*}
Where in the last equality, the degree $2$ bubble reversed its orientation and created a minus sign according to the infinite Grassmannian relation~\eqref{eq:infgrassmanian1neg},~\eqref{eq:infgrassmanian2pos}.

If we are to obtain an action of $\LLn[1]$ on the cup morphism, that is not given in a piecewise manner and does not depend on how we pass from $\mathfrak{sl}_2$ weights to $\mathfrak{gl}_2$, we must demand that both actions coming from different ladderizations agree. For that to hold true, we must have that 
\begin{align*}
    \lambda_1 + s - 1 &= -\mu_1 - s \\
    \lambda_1 + \mu_1 & = 1-2s 
\end{align*}
Repeating the process for higher differentials $\LLn[i]$, one sees that $s$ must be equal to $\frac{1}{2}$ and $\lambda_n + \mu_n = 0$, using the fact that these are Witt sequences. These are all the constraints we obtain in addition to the ones above.

If we look at other ladderizations for the image of the cup morphism under $\Phi$ such as the one below, we obtain no additional conditions on the parameters as it is straightforward to check.
\begin{align*}
    \LLn[1]  \circ \Phi &{\left(\CUPweb[1][<][0][\lambda][black][0][0]  \right)} = \LLn[1] \left(\cupEFfoam[0.5][0][0][0][0][0] \right) \\
\end{align*}

We will now look at a ladderization of the crossing for $\lambda > 0$, we may compute that
\begin{align*}
    \Phi^{-1} \circ \LLn[1]  \circ \Phi &\left(\Xwebwithtext[1][0][0][0][0][][] \right) =  \Phi^{-1} \circ \LLn[1] \left( \crossingEEfoam[.5] \right) 
\\
&= (\lambda_1 - \mu_1 + 1) \Xwebwithtext[1][0][0][][0][][]  +  (-\lambda_1 + \mu_1 + 1) \Xwebwithtext[1][0][0][0][][][] + ( - \lambda_1 + \mu_1) \quad \TWOweb
\end{align*}

Agreeing with the assignment in Theorem~\ref{main}. Checking the action on other cups and caps is similar, and deducing the action on upward, possibly dotted strands is straightforward.
\end{proof}

\begin{remark}
    The above lemma and its proof show that the action on the Karoubi envelope $\mathsf{Kar}({\Ucat(\mathfrak{g})})$ (that is the thick calculus) can be read from the action on foams via categorical Skew Howe duality and the presented choice of scalars.
\end{remark}

\section{Witt actions and trace decategorifications}
In this section, we restrict to $\mathfrak{g}$ of ADE type and work over a field of characteristic $0$.  
In~\cite{Beliakova_2016,Lau-trace3} it is shown that the trace decategorification $\mathrm{Tr}(\Ucat(\mathfrak{g}))$ is isomorphic to the universal enveloping algebra of the current algebra $U(\mathfrak{g}[t])$.  Furthermore, traces of  natural 2-representations of $\Ucat$ become representations for $U(\mathfrak{g}[t])$.  In particular, the traces of cyclotomic quotients $\mathcal{U}_Q^{\lambda,\ast}$ indexed by a highest weight $\lambda$ become local Weyl modules $W_\lambda^{loc}$ for $U(\mathfrak{g}[t])$, while the deformed cyclotomic quotients $\check{\mathcal{U}}_Q^{\lambda,\ast}$ from \cite{RouQH,Webster_2017} become global Weyl modules $W_{\lambda}$ under trace dectegorification~\cite[Theorem A.]{Lau-trace3}. 
The action of $\ourwitt$ on $\Ucat$ from Theorem~\ref{main} induces an action on $\check{\mathcal{U}}_Q^{\lambda,\ast}$.

Kac shows~\cite{kac1990infinite} that the Witt algebra acts as derivations on the loop algebra $\mf{g}[t,t^{-1}]$ via:
\begin{equation} \label{Witt-Kac}
   \LLn \cdot (x \otimes t^m) = - mx \otimes t^{n+m}, \quad x \in \mf{g}. 
\end{equation} 
Note that this action, when restricted to the polynomial part, agrees with the action of $\ourwitt$ on the polynomials from Lemma~\ref{actiononpoly}.
This action restricts to an action of the positive Witt algebra $\ourwitt$ on the current algebra $U(\mathfrak{g}[t])$.      

\begin{proposition}
In characteristic $0$, the trace decategorification from~\cite{Lau-trace3} $\mathrm{Tr} \maps \Ucat \to \mathrm{Tr}(\Ucat)\cong U(\mathfrak{g}[t])$ intertwines the 
the positive Witt algebra $\ourwitt$ action from Theorem~\ref{main}  with the action from~\eqref{Witt-Kac} making the diagram
       \begin{figure}[h!]
        \centering
\begin{tikzcd}
\Ucat(\mathfrak{g}) \arrow[r, "\LLn"] \arrow[d, "\mathrm{Tr}" description] & \Ucat(\mathfrak{g}) \arrow[d, "\mathrm{Tr}" description] \\
{U(\mathfrak{g}[t])} \arrow[r, "\LLn"]                                     & {U(\mathfrak{g}[t])}                            
\end{tikzcd}
    \end{figure}
commute.  
\end{proposition}

\begin{proof}
This is an immediate consequence of the fact that the  trace decategorification  of the action described in Lemma~\ref{actiononpoly} agreeing with the action from~\eqref{Witt-Kac} on the polynomial part of $U(\mathfrak{g}[t])$.  
\end{proof}

While the derivations on the current algebra are somewhat intuitive, our result suggests that even this natural action can be understood as having representation-theoretic origins, all arising from the higher structure of categorified quantum groups.

\subsection*{Further directions}
We would like to point out that our actions of $\LLn[1]$ do not recover all the degree $2$ derivations on $\Ucat(\mathfrak{sl}_2)$ as presented in~\cite{Elias_2016}. One should be able to enlarge the parameter space of our action, by classifying $\LLn[2]$. This might recover more of the actions of $\ourwitt$ on foams.

The authors would be interested in extending the action of $\ourwitt$ to an action of the whole Witt algebra or to an action of the Virasoro algebra by adding $\LLn$ for $n < -1$. 

Another point of interest is studying whether we can recover actions on  $2$-representations as coming from an action on the categorified quantum groups, as we were able to do for foams or Schur quotients with our choice of scalars. In particular, could this be related to known actions on the cyclotomic quotients or their trace decategorifications, which are known to be Weyl modules~\cite{Lau-trace3, Lau-trace2}.

%

\providecommand{\bysame}{\leavevmode\hbox to3em{\hrulefill}\thinspace}
\providecommand{\MR}{\relax\ifhmode\unskip\space\fi MR }
\providecommand{\MRhref}[2]{%
  \href{http://www.ams.org/mathscinet-getitem?mr=#1}{#2}
}
\providecommand{\href}[2]{#2}

\end{document}